\documentclass[final]{amsproc}
\usepackage{amssymb}
\usepackage{showlabels}
\usepackage[T1]{fontenc}
\usepackage{ifdraft}
\usepackage{amsfonts}
\usepackage{enumerate}
\usepackage{color}
\usepackage[textwidth=15cm, textheight=21cm]{geometry}

\theoremstyle{plain}
\newtheorem{theorem}{Theorem}

\newtheorem{corollary}[theorem]{Corollary}

\newtheorem{definition}[theorem]{Definition}

\newtheorem{lemma}[theorem]{Lemma}

\newtheorem{proposition}[theorem]{Proposition}

\theoremstyle{definition}
\newtheorem{example}[theorem]{Example}
\theoremstyle{remark}
\newtheorem{remark}[theorem]{Remark}
\numberwithin{equation}{section}
\numberwithin{theorem}{section}
\DeclareMathOperator{\R}{\mathbb R}
\DeclareMathOperator{\cM}{\mathcal{M}}

\DeclareMathOperator{\N}{\mathbb N}
\DeclareMathOperator{\supp}{supp}
\DeclareMathOperator{\col}{col}

\DeclareMathOperator{\diam}{diam}
\DeclareMathOperator{\diag}{Diag}
\DeclareMathOperator{\loc}{loc}
\DeclareMathOperator{\dimL}{\underline{\dim}_{A}}
\DeclareMathOperator{\dimA}{\overline{\dim}_{A}}
\DeclareMathOperator{\dimqL}{\underline{\dim}_{qA}}
\DeclareMathOperator{\dimqA}{\overline{\dim}_{qA}}
\DeclareMathOperator{\dimAscottheta}{\overline{\dim}_A^{\,=\theta}}
\DeclareMathOperator{\dimLscottheta}{\underline{\dim}_A^{\,=\theta}}
\DeclareMathOperator{\dimLscotthetaj}{\underline{\dim}_A^{\,=\theta_j}}

\DeclareMathOperator{\dimAcantheta}{\overline{\dim}_{\,A}^{\,\leq\theta}}
\DeclareMathOperator{\dimLcantheta}{\underline{\dim}_{\,A}^{\,\leq\theta}}

\DeclareMathOperator{\dimAscotpsi}{\overline{\dim}_A^{\,=\psi}}
\DeclareMathOperator{\dimLscotpsi}{\underline{\dim}_A^{\,=\psi}}

\DeclareMathOperator{\dimAscotbeta}{\overline{\dim}_A^{\,=\beta}}
\DeclareMathOperator{\dimLscotbeta}{\underline{\dim}_A^{\,=\beta}}

\begin{document}
\title[Lower Assouad dimension]{Lower Assouad Dimension of Measures and
Regularity\ifdraft{\\\today}{} }
\author{Kathryn E. Hare}
\address{Kathryn E. Hare\\
Department of Pure Mathematics\\
University of Waterloo\\
Waterloo, Canada\\
N2L 3G1}
\thanks{The first author's research was supported by NSERC 2016-03719. The
second author was supported in part by NSERC 2016-03719, NSERC 2014-03154
and the University of Waterloo.}
\author{Sascha Troscheit}
\address{Sascha Troscheit\\
Department of Pure Mathematics\\
University of Waterloo\\
Waterloo, Canada\\
N2L 3G1}
\date{\today }

\begin{abstract}
In analogy with the lower Assouad dimensions of a set, we study the lower
Assouad dimensions of a measure. As with the upper Assouad dimensions, the
lower Assouad dimensions of a measure provide information about the extreme
local behaviour of the measure. We study the connection with other
dimensions and with regularity properties. In particular, the quasi-lower
Assouad dimension is dominated by the infimum of the measure's lower local
dimensions. Although strict inequality is possible in general, equality
holds for the class of self-similar measures of finite type. This class
includes all self-similar, equicontractive measures satisfying the open set
condition, as well as certain \textquotedblleft
overlapping\textquotedblright\ self-similar measures, such as Bernoulli
convolutions with contraction factors that are inverses of Pisot numbers.

We give lower bounds for the lower Assouad dimension for measures arising
from a Moran construction, prove that self-affine measures are uniformly
perfect and have positive lower Assouad dimension, prove that the Assouad
spectrum of a measure converges to its quasi-Assouad dimension and show that
coincidence of the upper and lower Assouad dimension of a measure does not
imply that the measure is $s$-regular.
\end{abstract}

\maketitle

\section{Introduction}

The upper and lower Assouad dimensions of a metric space provide
quantitative information about the extreme local geometry of the set. The
analogous notion of the Assouad dimensions of a measure also quantifies, in
some sense, the extreme local behaviour of the measure. These dimensions
were extensively studied by K\"{a}enm\"{a}ki et al.,\ in \cite{KL} and \cite%
{KLV}, and Fraser and Howroyd, in \cite{FH}, where they were called the
upper and lower regularity dimensions. It was shown that the upper Assouad
dimension of a measure is finite if and only if the measure is doubling,
while the lower Assouad dimension is positive if and only if the measure is
uniformly perfect. K\"{a}enm\"{a}ki et al.\ focused their investigations on
doubling measures supported on uniformly perfect complete metric spaces,
whereas Fraser and Howroyd computed the upper Assouad dimension for a large
class of examples, as well as establishing links to other notions of
regularity. As many interesting measures are not doubling, such as is
typically the case for self-similar measures that fail the open set
condition, the weaker notion of the quasi-Assouad dimension of a measure is
more appropriate and was studied in \cite{HHT}. There it was shown, for
example, that self-similar measures that are sufficiently regular (but not
necessarily satisfying the open set condition), not only have finite
quasi-upper Assouad dimension, but in fact this dimension coincides with the
maximal local dimension of the measure.

In this paper, we investigate the lower Assouad dimension for measures and
introduce the quasi-lower Assouad dimension. The (quasi-) lower Assouad
dimension of a measure is easily seen to be dominated by the (quasi-) lower
Assouad dimension of the support of the measure. It is also dominated by the
infimum of the lower local dimensions (and hence the Hausdorff dimension) of
the measure. Although these dimensions are equal for self-similar measures
satisfying the strong separation condition, in general all the
aforementioned inequalities can be strict. We give various examples to show
this. We also give an example to show that equality of the upper and lower
Assouad dimensions does not imply $s$-regularity of the measure. In analogy
with what was shown for sets in \cite{CWC} and \cite{FHHTY}, we prove that
the quasi-lower and quasi-upper Assouad dimensions of measures can be
recovered from the Assouad dimension spectrum of a measure under the
assumption that the measure is quasi-doubling, i.e., has finite quasi-upper
Assouad dimension. These results can all be found in Sections \ref{Def} and %
\ref{sp}. In the appendix, we simplify the proof given in \cite{CWC} that
the quasi-lower Assouad dimension of a doubling metric space is the limit of
the dimension spectrum and remove their assumption that the metric space is
uniformly perfect.

In Section \ref{up} we establish a lower bound on the lower Assouad
dimension for uniformly perfect measures and show that certain Moran
constructions, such as self-similar and self-affine measures, have positive
lower Assouad dimension. For these sets, we give a lower bound for the
dimension in terms of the parameters of the Moran construction. We also
calculate the (quasi-) lower Assouad dimension of Bedford-McMullen carpets.

In Section \ref{FT} we prove the equality of the quasi-lower Assouad
dimension with the infimum of the set of lower local dimensions for
self-similar measures of finite type. This class of measures includes
equicontractive, self-similar measures satisfying the open set condition, as
well as certain measures that only satisfy the weak separation condition,
such as Bernoulli convolutions with contraction factor the inverse of a
Pisot number. Our proof is constructive; we exhibit a sequence of points
such that the lower local dimension of the measure at these points tends to
the quasi-lower Assouad dimension of the measure.

A measure is said to be $L^{p}$-improving if it acts by convolution as a
bounded map from $L^{2}$ to $L^{p}$ for some $p>2$. It is known that $L^{p}$%
-improving measures have positive Hausdorff dimension, thus it is natural to
ask if they must also have positive lower Assouad dimension. In Section \ref%
{LI}, examples are given to show that even the quasi-lower Assouad dimension
of an $L^{p}$-improving measure can be zero, although its local dimensions
must be bounded away from zero. In fact, we show that there exist measures
whose Fourier transform is $p$-summable for some $p<\infty ,$ with zero
quasi-lower Assouad dimension.

\section{Definitions and basic properties of the lower Assouad type
dimensions\label{Def}}

\subsection{Assouad dimensions of sets}

Given a compact metric space $X$, we write $N_{r}(E)$ for the least number
of sets of diameter at most $r$ that are required to cover $E\subseteq X$.
Given $\delta >0,$ let 
\begin{eqnarray*}
\overline{h}(\delta ) &=&\inf \left\{ \alpha :(\exists c,c_{2}>0)(\forall
0<r\leq R^{1+\delta }\leq c_{1})\ \sup_{x\in E}N_{r}(B(x,R)\cap E)\leq
c_{2}\left( \frac{R}{r}\right) ^{\alpha }\right\} , \\
\underline{h}(\delta ) &=&\sup \left\{ \alpha :(\exists
c_{1},c_{2}>0)(\forall 0<r\leq R^{1+\delta }\leq c_{1})\ \inf_{x\in
E}N_{r}(B(x,R)\cap E)\geq c_{2}\left( \frac{R}{r}\right) ^{\alpha }\right\} .
\end{eqnarray*}%
The upper Assouad and lower Assouad dimensions of $E$ are given by 
\begin{equation*}
\dimA\,E=\overline{h}(0),\;\dimL E=\underline{h}(0),
\end{equation*}%
while the quasi-upper Assouad and quasi-lower Assouad dimensions are given
by 
\begin{equation*}
\dimqA\,E=\lim_{\delta \rightarrow 0}\overline{h}(\delta )\text{, }\dimqL %
E=\lim_{\delta \rightarrow 0}\underline{h}(\delta ).
\end{equation*}

\subsection{Assouad dimensions of measures}

By a measure we will mean a Borel probability measure on $X$ with compact
support. The analogue of the upper Assouad and lower Assouad dimensions for
measures was studied in \cite{FH}, \cite{KL} and \cite{KLV} (where they were
called upper and lower regularity dimensions). The analogue of the
quasi-upper Assouad dimension for measures was introduced in \cite{HHT}.
This paper is primarily concerned with the (quasi)-lower Assouad dimension
for measures.

Given a measure $\mu $ and $\delta \geq 0$, set 
\begin{equation*}
\overline{H}(\delta ) =\inf \left\{ s:(\exists c_{1},c_{2}>0)(\forall
0<r\leq R^{1+\delta }\leq c_{1})\sup_{x\in \supp\mu }\frac{\mu (B(x,R))}{\mu
(B(x,r))}\leq c_{2}\left( \frac{R}{r}\right) ^{s}\right\}
\end{equation*}
and 
\begin{equation*}
\underline{H}(\delta ) =\sup \left\{ s:(\exists c_{1},c_{2}>0)(\forall
0<r\leq R^{1+\delta }\leq c_{1})\inf_{x\in \supp\mu }\frac{\mu (B(x,R))}{\mu
(B(x,r))}\geq c_{2}\left( \frac{R}{r}\right) ^{s}\right\} .
\end{equation*}

\begin{definition}
The \textbf{upper Assouad }and \textbf{lower Assouad dimensions} of $\mu $
are given by%
\begin{equation*}
\dimA\,\mu =\overline{H}(0),\;\dimL\,\mu =\underline{H}(0).
\end{equation*}%
The \textbf{quasi-upper Assouad }and \textbf{quasi-lower Assouad dimension}
of $\mu $ are given by%
\begin{equation*}
\dimqA\mu =\lim_{\delta \rightarrow 0}\overline{H}(\delta ),\; \dimqL\mu
=\lim_{\delta \rightarrow 0}\underline{H}(\delta ).
\end{equation*}
\end{definition}

\begin{remark}
We note that these dimensions are known under various names and many
different notations are in common use. The upper Assouad dimension is often
referred to as \emph{the} Assouad dimension, the lower Assouad dimension
sometimes simply as lower dimension, and the measure theoretic versions as
the upper and lower regularity dimensions. We have opted to use a bar to
denote upper or lower Assouad dimension instead of $\dim _{A}$ and $\dim
_{L} $, (as $\dim _{L}$ is sometimes used to refer to the Lyapunov dimension
of a measure).
\end{remark}

\subsection{Relationships between these dimensions}

It is clear from the definitions that 
\begin{equation*}
0\leq \dimL E\leq \dimqL E\leq \dimqA E\leq \dimA E\leq \infty
\end{equation*}%
and 
\begin{equation*}
0\leq \dimL\mu \leq \dimqL\mu \leq \dimqA\mu \leq \dimA\mu \leq \infty.
\end{equation*}%
It was shown in \cite{FH} and \cite{HHT} that 
\begin{equation*}
\dimA\mu \geq \dimA\supp\mu \quad\text{and}\quad\;\dimqA\mu \geq \dimqA\supp%
\mu .\text{ }
\end{equation*}%
It is known that $\dimA\mu <\infty $ if and only if $\mu $ is doubling,
meaning there is a constant $C>0$ such that 
\begin{equation}
\mu (B(x,R))\geq C\mu (B(x,2R))\text{ for all }x,R.  \label{doubling}
\end{equation}%
See \cite{FH} for a proof.

Recall that the lower local dimension of $\mu $ at $x$ is defined as%
\begin{equation*}
\underline{\dim }_{\loc}\mu (x)=\liminf_{r\rightarrow 0}\frac{\log \mu
(B(x,r))}{\log r},
\end{equation*}%
with the upper local dimension, $\overline{\dim }_{\loc}\mu (x),$ defined
similarly but with $\lim \sup $ replacing $\liminf $. Fraser and Howroyd in 
\cite{FH} also showed that%
\begin{equation*}
\dimqA\mu \geq \sup_{x\in \supp\mu }\{\overline{\dim }_{\loc}\mu (x)\}.
\end{equation*}

Similar relations hold for the (quasi-)lower Assouad dimensions.

\begin{proposition}
(i) If $\mu $ is a doubling measure, then 
\begin{equation*}
\dimL\mu \;\leq\;\dimL\supp\mu \quad\text{ and }\quad\dimqL\mu \;\leq\;\dimqL%
\supp\mu .
\end{equation*}

(ii) For any measure $\mu ,$ 
\begin{equation*}
\dimL\mu \;\leq \;\dimqL\mu \;\leq \;\inf_{x\in \supp\mu }\{\underline{\dim }%
_{\loc}\mu (x)\}\;\leq \;\dim _{H}\mu .
\end{equation*}

(iii) If $\mu $ is a self-similar measure associated with an IFS that
satisfies the strong separation condition, then 
\begin{equation*}
\dimL\mu =\inf_{x}\{\dim _{\loc}\mu (x)\}.
\end{equation*}
\end{proposition}

\begin{proof}
(i) The fact that $\dimL\mu \;\leq \;\dimL\supp\mu $ was observed in \cite%
{KL}. To see that $\dimqL\mu \;\leq \;\dimqL\supp\mu ,$ let $t=\dimqL\supp%
\mu $ and $C$ be the doubling constant of (\ref{doubling}). For any $%
\varepsilon >0$ and suitable $\delta >0$, there are $x_{i}\in \supp\mu $, $%
R_{i}\rightarrow 0$ and $r_{i}\leq R_{i}^{1+\delta }$ such that $%
N_{r_{i}}(B(x_{i},R_{i})\cap \supp\mu )\leq (R_{i}/r_{i})^{t+\varepsilon }$.
Together with the doubling property, this implies 
\begin{eqnarray*}
\mu (B(x_{i},2R_{i})\cap \supp\mu ) &\leq &C^{-1}\mu (B(x_{i},R_{i})\cap %
\supp\mu )\vphantom{\max_{y\in B(x_{i},R_{i})}} \\
&\leq &C^{-1}N_{r_{i}}(B(x_{i},R_{i})\cap \supp\mu )\max_{y\in
B(x_{i},R_{i})}\mu (B(y,r_{i})\cap \supp\mu ) \\
&\leq &C^{-1}(R_{i}/r_{i})^{t+\varepsilon }\mu (B(y_{i},r_{i})\cap \supp\mu )
\end{eqnarray*}%
for a suitable $y_{i}\in B(x_{i},R_{i})$. Now $B(y_{i},R_{i})\subseteq
B(x_{i},2R_{i})$ and thus%
\begin{equation*}
\frac{\mu (B(y_{i},R_{i})\cap \supp\mu )}{\mu (B(y_{i},r_{i})\cap \supp\mu )}%
\leq \frac{\mu (B(x_{i},2R_{i})\cap \supp\mu )}{\mu (B(y_{i},r_{i})\cap \supp%
\mu )}\leq C^{-1}(R_{i}/r_{i})^{t+\varepsilon }.
\end{equation*}%
That suffices to show $\dimqL\mu \leq t$. A similar argument shows $\dimL\mu
\leq \dimL\supp\mu .$

(ii) The only new statement here is the inequality $\dimqL\mu \leq \inf_{x}\{%
\underline{\dim }_{\loc}\mu (x)\}$ and this follows in the same manner as 
\cite[Proposition 2.4]{HHT}.

(iii) The proof of this is essentially the same as given in \cite[Theorem 2.4%
]{FH} for the fact that $\dimA\mu =\sup_{x}\{\dim _{\loc}\mu (x)\}$.
\end{proof}

\begin{remark}
In \cite[Proposition 4.2]{HHT} it is shown that if $\dimqA\mu <\infty ,$
then for each $\varepsilon >0$ there is a constant $c>0$ such that $\mu
(B(x,R))\geq cR^{\varepsilon }\mu (B(x,2R))$ for all $x,R$. The reader can
check that this weaker condition suffices to ensure $\dimqL\mu \leq \dimqL%
\supp\mu .$
\end{remark}

\begin{remark}
Strict inequalities are possible between all these dimensions. Indeed, in 
\cite[Example 2.3]{HHT} it is explained how to construct examples with $%
\dimqA\supp\mu <\dimA\supp\mu <\dimqA\mu <\dimA\mu $ and $\dimqA\supp\mu <%
\dimqA\mu <\dimA\supp\mu <\dimA\mu $. It is easy to modify these to produce
analogous examples for the lower Assouad dimensions. In particular, one can
have $\dimL\mu =0,\dimA\mu =\infty ,$ but $0<\dimqL\mu <\dimqA\mu <\infty $.
Below we give an example where $\dimqL\mu <\inf_{x}\{\underline{\dim }_{\loc%
}\mu (x)\}$. Another example is Example~\ref{LpI}. In Section \ref{FT} we
prove that the equality does hold for a large class of self-similar
measures, which need not satisfy the open set condition.
\end{remark}

\begin{example}
\label{qLneLoc}A measure $\mu $ on $\mathbb{R}$ with $\dimqL\mu =0$ and $%
\inf_{x}\{\underline{\dim }_{\loc}\mu (x)\}=1$: We construct a probability
measure $\mu $ with support $[0,1]$ defined iteratively on the dyadic
intervals. Label the dyadic intervals of length $2^{-n}$ (step $n) $ from
left to right as $I_{n}^{(i)},i=1,\dots,2^{n}$, so $I_{n}^{(1)},$ $%
I_{n}^{(2)}$ are the two descendants of $I_{n-1}^{(1)}$, for example. Let $%
\{n_{j}\}$ be an integer sequence with $n_{j+1}\geq 3n_{j}$ . Choose a
sequence $1/2\leq q_{j}\uparrow 1$ and put $%
t_{j}=q_{j}^{-n_{j}}2^{-(1+n_{j})}$. Assuming $\mu $ has been defined on the
dyadic intervals of step $n-1$, we define $\mu $ on the dyadic intervals of
step $n$ in the following fashion:%
\begin{eqnarray*}
\mu (I_{n}^{(1)}) &=&t_{j}\mu (I_{n-1}^{(1)})\quad\text{ and }\quad\mu
(I_{n}^{(2)})=(1-t_{j})\mu (I_{n-1}^{(1)})\quad\text{if}\quad n=n_{j}, \\
\mu (I_{n}^{(1)}) &=&q_{j}\mu (I_{n-1}^{(1)})\quad\text{ and }\quad\mu
(I_{n}^{(2)})=(1-q_{j})\mu (I_{n-1}^{(1)})\quad\text{if}\quad
n=n_{j}+1,\dots,2n_{j}.
\end{eqnarray*}%
All other dyadic intervals of step $n$ will have measure $1/2$ that of their
parent interval.

We even have \underline{$h$}$(1/2)=0,$ and thus $\dimqL\mu =0$, because 
\begin{equation*}
\frac{\mu (B(0,2^{-n_{j}}))}{\mu (B(0,2^{-2n_{j}}))}=\frac{1}{q_{j}^{n_{j}}}%
=\left( \frac{2^{-n_{j}}}{2^{-2n_{j}}}\right) ^{t}
\end{equation*}%
for $t=-\log q_{j}/\log 2$ and $t\rightarrow 0$ as $q_{j}\rightarrow 1$.

To see that $\underline{\dim }_{\loc}\mu (x)\geq 1$ for all $x\in \supp\mu ,$
we first consider $x\neq 0$. Choose $N_{0}$, depending on $x$ such that $%
x>4\cdot 2^{-N_{0}}$. If $2^{-(n+1)}<r\leq 2^{-n}$ for $n\geq N_{0},$ then $%
B(x,r)$ is contained in the union of four consecutive dyadic intervals of
length $2^{-n}$, none of which intersect the two left-most intervals of step 
$N_{0}$. Thus 
\begin{equation*}
\mu (B(x,r))\leq 4\cdot 2^{N_{0}-n}\max \left( \mu (I_{N_{0}}^{(i)}):i\geq
3\right) =C2^{N_{0}-n},
\end{equation*}%
so 
\begin{equation*}
\frac{\log \mu (B(x,r))}{\log r}\geq \frac{\log C2^{N_{0}-n}}{\log 2^{-(n+1)}%
}\rightarrow 1\text{ as }n\rightarrow \infty .
\end{equation*}%
Finally, consider $x=0$. The choice of $t_{j}$ ensures that $\mu
(B(0,2^{-n}))\leq 2^{-n}$ for all $n$ and that certainly implies $\underline{%
\dim }_{\loc}\mu (0)\geq 1$. That completes the proof.
\end{example}

\subsection{Lower dimension and regularity}

A measure $\mu $ is called $s$-regular if there exists a uniform constant $%
c>0$ such that 
\begin{equation*}
c^{-1}r^{s}\leq \mu (B(x,r))\leq cr^{s}
\end{equation*}%
for all $x\in \supp\mu $ and $0<r<\diam\supp\mu $. It is easy to show from
the definitions that if $\mu $ is $s$-regular then $\dimL\mu =\dimA\mu =s$,
see e.g.\ \cite{KL} and \cite{KLV}. However, it is not true that coinciding
lower and upper Assouad dimension implies $s$-regularity, as the following
example illustrates.

\begin{example}
Let $M_{v}$ be the collection of triadic intervals labelled by finite words
on the letters $\{0,1,2\}$. We construct a finite measure $\mu $ on $[0,1]$
as follows: 
\begin{eqnarray}
\mu (M_{v}) &=&(k+1)3^{-(k+1)}\text{ if }v=1^{(k)}0\text{ or }v=1^{(k)}2, 
\notag \\
\mu (M_{1^{(k)}jv}) &=&\frac{k+1}{3^{k+1+l}}\text{ if }j\in \{0,2\}\text{
and }v\in \{0,1,2\}^{l},  \label{eqref} \\
\mu (M_{1^{(k)}}) &=&2\sum_{i=k+1}^{\infty }\frac{i}{3^{i}}=3^{-k}(3/2+k). 
\notag
\end{eqnarray}%
One can easily check that $\mu $ is well defined and upon normalizing by $%
\mu ([0,1])=2\sum_{i=0}^{\infty }\frac{i+1}{3^{i+1}}=\frac{3}{2},$ we obtain
a probability measure.

We now estimate the ratio between any triadic interval and its descendants.
Consider $M_{v}$ and $M_{vw}$ for $v\in \{0,1,2\}^{k}$ and $w\in
\{0,1,2\}^{l}$, where $l\geq 1$. If $v\neq 1^{(k)}$, then $\mu (M_{v})/\mu
(M_{vw})=3^{l}$, using \eqref{eqref}. If, however, $v=1^{(k)}$, then 
\begin{equation}
\frac{k+1}{3^{k+l}}=\mu (M_{v0^{(l)}})\leq \mu (M_{vw})\leq \mu
(M_{1^{(k+l)}})=3^{-(k+l)}(3/2+(k+l))  \label{ineq1}
\end{equation}%
Note also that for $j\in \{0,2\}$ and $k\geq 1$, 
\begin{equation}
\frac{\mu (M_{1^{(k)}})}{\mu (M_{1^{(k-1)}j})}=\frac{3^{-k}(3/2+k)}{k3^{-k}}=%
\frac{3/2+k}{k}\leq \frac{5}{2}.  \label{ineq2}
\end{equation}%
The inequalities \eqref{ineq1} and \eqref{ineq2} show that neighbouring
triadic intervals of the same length differ by at most a factor of $5/2$.

Now let $J\subseteq I\subseteq \lbrack 0,1]$ be intervals. Write $k$ and $l$
for the unique integers satisfying $3^{-(k-1)}\leq \diam I\leq 3^{-(k-2)}$
and $3^{-(k+l-1)}\leq \diam J\leq 3^{-(k+l-2)}$. Thus $I$ contains a triadic
interval of length $3^{-k}$ and is contained within $10$ intervals of length 
$3^{-k}$. Analogously, $J$ contains an interval of length $3^{-(k+l)}$ and
is contained in $10$ intervals of the same length. We can therefore find $%
M_{v}$ and $M_{vw}$, $v\in \{0,1,2\}^{k}$, $w\in \{0,1,2\}^{l}$ such that 
\begin{equation*}
\mu (M_{v})\leq \mu (I)\leq \left( \frac{5}{2}\right) ^{10}\mu (M_{v})\quad 
\text{and}\quad \mu (M_{vw})\leq \mu (J)\leq \left( \frac{5}{2}\right)
^{10}\mu (M_{vw}),
\end{equation*}%
and hence 
\begin{equation*}
\frac{\mu (I)}{\mu (J)}\sim \frac{\mu (M_{v})}{\mu (M_{vw})}
\end{equation*}%
where $\sim $ denotes uniform comparability. But 
\begin{equation*}
\frac{\mu (M_{v})}{\mu (M_{vw})}\leq \max \left\{ 3^{l},\,\frac{3^{-k}(3/2+k)%
}{(k+1)3^{-(k+l)}}\right\} =3^{l}\frac{3/2+k}{k+1}\leq \frac{3}{2}3^{l}
\end{equation*}%
and 
\begin{equation*}
\frac{\mu (M_{v})}{\mu (M_{vw})}\geq \min \left\{ 3^{l},\,\frac{3^{-k}(3/2+k)%
}{3^{-(k+l)}(3/2+(k+l))}\right\} =3^{l}\frac{3/2+k}{3/2+(k+l)}\geq 3^{l}%
\frac{5/2}{5/2+l}.
\end{equation*}%
So 
\begin{equation*}
\left( \frac{5}{2}\right) ^{11}3^{l}\geq \frac{\mu (I)}{\mu (J)}\geq \left( 
\frac{5}{2}\right) ^{-10}\frac{5/2}{5/2+l}3^{l}.
\end{equation*}%
Further, $(\diam I)/(\diam J)\sim 3^{l}$ and so for every $\delta >0$ there
exists $C>0$ such that 
\begin{equation*}
C\frac{\diam I}{\diam J}\geq \frac{\mu (I)}{\mu (J)}\geq C^{-1}\left( \frac{%
\diam I}{\diam J}\right) ^{1-\delta }.
\end{equation*}%
In particular this holds for $I=B(x,R)$ and $J=B(x,r)$ and so the upper and
lower Assouad dimension of $\mu $ is $1$. But $\mu (B(1/2,\,3^{-k}))=\mu
(M_{1^{(k)}})=3^{-k}(3/2+k)$ and there is no constant $K>0$ such that $\mu
(B(x,r))\leq Kr$, so $\mu $ is not $1$-regular. Since it cannot be $s$%
-regular for any $s\neq 1$, the measure $\mu $ is not $s$-regular for any $%
s\geq 0$.
\end{example}

\section{Uniformly perfect measures\label{up}}

Analogous to the metric space properties, it is known that a measure has
positive lower Assouad dimension if and only if it is uniformly perfect,
c.f.~\cite{KL}. We exhibit a general Moran type construction of a measure
that has positive lower Assouad dimension and give a lower bound on the
lower Assouad dimension in terms of the Moran construction data. We show
that many commonly considered fractal measures satisfy the construction
constraints. In particular, self-affine measures are seen to have positive
lower Assouad dimension, and hence are uniformly perfect, as long as they
are not a degenerate point mass.

\subsection{Characterizing positive lower Assouad dimension}

\begin{definition}
Let $\mu $ be a compactly supported Borel probability measure. If there
exist positive constants $c,\gamma $ such that 
\begin{equation}
\mu (B(x,R)\setminus B(x,cR))\geq \gamma \mu (B(x,R))  \label{unifperfect}
\end{equation}%
for all $x\in\supp\mu$ and $R\leq \diam (\supp\mu)$, we say that $\mu $ is 
\textbf{uniformly perfect}\footnote{%
This condition is also known as ``inverse doubling''.}.
\end{definition}

Of course \eqref{unifperfect} is equivalent to the statement that 
\begin{equation}  \label{eq:rewritten}
\frac{\mu (B(x,R))}{\mu (B(x,cR))}\geq (1-\gamma )^{-1}.
\end{equation}%
We have opted to state our definition to mirror the metric space definition
of uniformly perfect, which states that a metric space is uniformly perfect
if for every centred ball, the annulus must be non-empty. From the
definition of uniformly perfect for measures it is immediate that the
support of a uniformly perfect measure must also be uniformly perfect.
However, the converse may not be true; it is possible to construct a measure
which is not uniformly perfect, but supported on an uniformly perfect set.;
c.f., Example \ref{qLneLoc} where the measure $\mu $ has support equal to $%
[0,1]$.

\begin{theorem}
\label{thm:lowerlower} Let $\mu $ be a compactly supported Borel probability
measure. Then $\dimL\mu >0$ if and only if $\mu $ is uniformly perfect. More
precisely, if $\mu $ is uniformly compact with positive constants $c,\gamma $
as in \eqref{unifperfect}, then $\dimL\mu \geq \log (1-\gamma )/\log c.$
\end{theorem}

\begin{proof}
First, assume $\mu $ is uniformly perfect. Let $c,\gamma $ be as (\ref%
{unifperfect}). For $r<R$, choose $n$ such that $c^{n-1}R>r\geq c^{n}R$.
Without loss of generality, $n\geq 2$ and repeatedly applying %
\eqref{eq:rewritten} gives 
\begin{eqnarray*}
\frac{\mu (B(x,R))}{\mu (B(x,r))} &\geq &\frac{\mu (B(x,R))}{\mu
(B(x,c^{n-1}R))}\geq \frac{\mu (B(x,R))}{\mu (B(x,cR))}\frac{\mu (B(x,cR))}{%
\mu (B(x,c^{2}R))}\cdot \cdot \cdot \frac{\mu (B(x,c^{n-2}R))}{\mu
(B(x,c^{n-1}R))} \\
&\geq &(1-\gamma )^{-(n-1)}\geq (1-\gamma )(1-\gamma )^{\log (R/r)/\log
c}=(1-\gamma )\left( \frac{R}{r}\right) ^{\frac{\log (1-\gamma )}{\log c}}.
\end{eqnarray*}%
Thus $\dimL\mu \geq \log (1-\gamma )/\log c>0$.

The other direction is straightforward and follows directly from the
definition.
\end{proof}

\subsection{Moran constructions}

Let $\Lambda =\{1,\dots ,N\}$ be a finite alphabet with $2\leq N<\infty $
letters and write $\Lambda ^{k}$ for all words of length $k$, $\Lambda
^{\ast }$ for the collection of all finite words including the empty word $%
\varepsilon _{0}$, and $\Lambda ^{\mathbb{N}}$ for all infinite words. A
countable subset $S\subseteq \Lambda ^{\ast }$ is called a \emph{section} if
for every long word $w\in \Lambda ^{\ast }$ there exists $u\in S$ and $v\in
\Lambda ^{\ast }$ such that $w=uv$, i.e.,\ every long enough word has an
ancestor in $S$. A section $S$ is \emph{minimal} if no proper subset of $S$
is a section.

For every word $v\in \Lambda ^{\ast }$, let $M_{v}\subset X$ be an arbitrary
set satisfying the following conditions:

\begin{enumerate}[(a)]

\item \label{cond:subset} $M_{vw}\subseteq M_v$ for all $v,w\in\Lambda^*$;

\item \label{cond:decr} $\max_{v\in\Lambda^k} \diam M_v\to0$ as $k\to\infty$;

\item \label{cond:notsmall} $\diam(M_{vj})\geq C_{1}\diam(M_{v})$ for all $%
v\in \Lambda ^{\ast }$, $j\in \Lambda $, and $C_{1}>0$ not depending on $v$
and $j$;

\item \label{cond:notbunching} for every $v\in \Lambda ^{\ast }$ there exist 
$i,j\in \Lambda $ such that $d(M_{vi},M_{vj})\geq C_{2}\diam(M_{v})$, where $%
C_{2}>0$ does not depend on $v,i,j$ and $d(A,B)$ denotes the distance of
sets $A$ and $B$.
\end{enumerate}

Finally, let $M$ be the $\limsup $ set of $\{M_{v}\}$: 
\begin{equation*}
M=\bigcap_{n=1}^{\infty }\bigcup_{k=n}^\infty\bigcup_{v\in \Lambda
^{k}}M_{v}.
\end{equation*}

Let $m$ be a weight function on the collection $\{M_v\,:\,v\in\Lambda^*\}$
satisfying the following conditions:

\begin{enumerate}[(A)]

\item $m(M_{\varepsilon_0})=1$;

\item $m(M_{v})=\sum_{i=1}^N m(M_{vi})$;

\item $m(M_{vi})\leq C_3 m(M_v)$ for some uniform $0<C_3<1$.
\end{enumerate}

For $E\subseteq X$, let $\Lambda ^{\ast }(E)$ be the collection of all words 
$v\in \Lambda ^{\ast }$ such that $M_{v}\cap E\neq \emptyset $. We let $\mu $
be the measure\footnote{%
Strictly speaking, $\mu $ is an outer measure, as proven in Lemma~\ref%
{thm:outermeasure}. We will consider $\mu $ as a set function, and when
using properties of measures we will assume, without further mention,
measurability of the sets being considered.} induced by the weight function.
In other words,\ writing $S$ for the collection of all minimal sections of $%
\Lambda ^{\ast }$, the measure $\mu $ is given by 
\begin{equation}
\mu (E)=\inf_{A\in S}\left\{ \sum_{v\in A^{\prime }}m(M_{v})\;:\;A^{\prime
}=A\cap \Lambda ^{\ast }(E)\right\} ,  \label{meas}
\end{equation}%
for all $E\subseteq X$. In particular, our conditions give $\supp\mu =M$.

\clearpage
\begin{lemma}
\label{thm:outermeasure} The set function $\mu$, as constructed above, is an
outer measure.
\end{lemma}

\begin{proof}
Clearly, $\Lambda ^{\ast }(\emptyset)=\emptyset$ and thus $\mu (\emptyset)=0$%
. For monotonicity, let $D\subseteq E\subseteq M$ and observe that for every 
$\epsilon >0$ there exists a section $A_{\varepsilon } $ such that 
\begin{equation*}
\mu (E)\leq \sum_{v\in A_{\varepsilon }^{\prime }}m(M_{v})\leq \mu
(E)+\varepsilon ,
\end{equation*}%
where $A_{\varepsilon }^{\prime }=A_{\varepsilon }\cap \Lambda ^{\ast }(E)$.
Now $D\subseteq E$ and so $\Lambda ^{\ast }(D)\subseteq \Lambda ^{\ast }(E)$%
. Therefore, 
\begin{equation*}
\mu (D)\leq \sum_{v\in A\cap \Lambda ^{\ast }(D)}m(M_{v})\leq \sum_{v\in
A_{\varepsilon }^{\prime }}m(M_{v})\leq \mu (E)+\varepsilon .
\end{equation*}%
Since $\varepsilon $ was arbitrary we obtain the required $\mu (D)\leq \mu
(E)$.

Finally, for countable subadditivity, let $E_{i}$, $i\in \mathbb{N}$, be a
sequence of subsets of $M$. Let $\varepsilon >0$ be arbitrary and define $%
\varepsilon _{i}=\varepsilon /2^{i}$. Let $A_{i}$ be a section such that 
\begin{equation*}
\mu (E_{i})\leq \sum_{v\in A_{i}^{\prime }}m(M_{v})\leq \mu
(E_{i})+\varepsilon _{i},
\end{equation*}%
where $A_{i}^{\prime }=A_{i}\cap \Lambda ^{\ast }(E_{i})$. Let $B^{\prime
\prime }=\bigcup A_{i}^{\prime }$ and let $B^{\prime }\subseteq B^{\prime
\prime }$ be a minimal subset, meaning that\ if $v\in B^{\prime },$ then
there does not exist non-empty $w\in \Lambda ^{\ast }$ such that $vw\in
B^{\prime }$. Note that $\bigcup A_{i}$ is a countable section, though not
necessarily minimal, and must contain a minimal section $B$ that contains $%
B^{\prime }$.

We now show that if $v\in B$ and $M_{v}\cap \bigcup E_{i}\neq \emptyset$,
then $v\in B^{\prime }$. So assume that for some $v\in B$ \ there exists $%
x\in M_{v}\cap \bigcup E_{i}$. Then there exists $j$ such that $x\in E_{j}$
and a coding $vw\in \Lambda ^{\mathbb{N}}$ such that $\bigcap_{i=1}^{\infty
}M_{(vw)\rvert _{i}}=x$. Since $A_{j}$ is a section there must exist $k$
such that $(vw)\rvert _{k}\in A_{j}$. Further, as $x\in M_{vw\rvert _{k}}$
we have $(vw)\rvert _{k}\in \Lambda ^{\ast }(\{x\})\subseteq \Lambda ^{\ast
}(E_{j})$ and so $(vw)\rvert _{k}\in A_{j}^{\prime }$ and $(vw)\rvert
_{k}\in B^{\prime \prime }$. Since $B^{\prime }$ is a minimal section it
must contain $(vw)\rvert _{l}$ for some $l\leq k$. We cannot have $l>|v|$ as
then $(vw)\lvert _{l}$ has the ancestor $v$ in $B$ and $B$ is not minimal.
Further, we cannot have $l<|v|$ for then $v\in B$ has an ancestor in $B$,
again breaking minimality. Hence $l=\lvert v\rvert $ and $v=(vw)\rvert
_{l}\in B^{\prime }$, as required.

We can now bound the measure of $\bigcup E_{i}$: 
\begin{eqnarray*}
\mu \left( \bigcup E_{i}\right) &\leq &\sum_{v\in B^{\prime }}m(M_{v})\leq
\sum_{v\in B^{\prime \prime }}m(M_{v})\leq \sum_{i\in \mathbb{N}}\sum_{v\in
A_{i}^{\prime }}m(M_{v}) \\
&\leq &\sum_{i\in \mathbb{N}}(\mu (E_{i})+\varepsilon _{i})=\sum_{i\in 
\mathbb{N}}\mu (E_{i})+\varepsilon .
\end{eqnarray*}%
Letting $\varepsilon \rightarrow 0$ gives the required subadditivity.
\end{proof}

Using this construction and Theorem \ref{thm:lowerlower} we can prove the
following theorem\footnote{%
Independently, Rossi and Shmerkin \cite[\S 4.2]{Rossi18} also proved that a
similar Moran construction is uniformly perfect.}.

\begin{theorem}
Let $\Lambda $ be a finite alphabet and $M_{v}$, $v\in \Lambda ^{\ast }$ and 
$m$ satisfy the conditions above. Then 
\begin{equation*}
\dimL\mu \geq \frac{\log (1-C_{3})}{\log C_{1}}>0
\end{equation*}%
and hence $\mu $ is uniformly perfect.
\end{theorem}

\begin{proof}
Let $x\in M$ and $R>0$ be arbitrary. We define 
\begin{equation*}
\mathcal{C}=\{v\in \Lambda ^{\ast }\,:\,\diam(M_{v})\leq C_{1}R\text{, }\diam%
(M_{v^{-}})>C_{1}R\text{, }M_{v}\subseteq B(x,R)\}.
\end{equation*}%
Note that, by definition, $\bigcup_{v\in \mathcal{C}}M_{v}\subseteq B(x,R)$.
Let $m$ be large enough that $C_{1}^{m}+C_{1}<1$. Choose $n$ so large such
that $e^{-n}\leq C_{1}^{m}$, so any $w\in \Lambda ^{\ast }$ for which $%
M_{w}\cap B(x,e^{-n}R)\neq \emptyset $ and $\diam M_{w}\leq e^{-n}R$ must
have an ancestor $w^{\prime }$ such that $\diam M_{w^{\prime }}\leq C_{1}R$
but $\diam M_{w^{\prime -}}>C_{1}R$. Therefore 
\begin{equation*}
d(x,y)<e^{-n}R+C_{1}R\leq (C_{1}^{m}+C_{1})R<R
\end{equation*}%
for all $y\in M_{w^{\prime }}$ and $M_{w^{\prime }}\subseteq B(x,R)$. Hence, 
$w^{\prime }\in \mathcal{C}$ and, in particular, every word in 
\begin{equation*}
\mathcal{B}=\{v\in \Lambda ^{\ast }\,:\,\diam(M_{v})\leq e^{-n}R\text{, }%
\diam(M_{v^{-}})>e^{-n}R\text{, }M_{v}\cap B(x,e^{-n}R)\neq \emptyset \}
\end{equation*}%
must have an ancestor in $\mathcal{C}$. Let $k$ be the maximal integer such
that $C_{2}C_{1}^{k+2}R>3e^{-n}R$ and temporarily fix $v\in \mathcal{C}$.
Note that for all $1\leq j\leq k$, there exist two words $\alpha ,\beta \in
\Lambda ^{j}$ such that $\diam M_{v\alpha }$, $\diam M_{v\beta }>e^{-n}R$
and further that $d(M_{v\alpha },M_{v\beta })>3e^{-n}R$. Hence, at most one
of $M_{v\alpha },M_{v\beta }$ can intersect $B(x,e^{-n}R)$ and for every $j$
there exists at least one $w_{j}\in \Lambda $ such that $M_{vz_{1}z_{2}\dots
z_{j-1}w_{j}}\cap B(x,e^{-n}R)=\emptyset $ where $z_{1}\neq w_{1}$, $%
z_{2}\neq w_{2}$, etc. Since $m(M_{w})<C_{3}m(M_{w^{-}})$, we further get 
\begin{equation*}
m\left( \bigcup_{i\in \Lambda \setminus w_{1}}M_{vi}\right) \leq
(1-C_{3})m(M_{v})
\end{equation*}%
and, inductively, for $W=(\Lambda \setminus w_{1})\times (\Lambda \setminus
w_{2})\times \dots \times (\Lambda \setminus w_{k_{i}})$, 
\begin{equation*}
m\left( \bigcup_{u\in W}M_{vu}\right) \leq (1-C_{3})^{k}m(M_{v}).
\end{equation*}%
Observe that by construction $M_{vw}\cap B(x,e^{-n}R)=\emptyset $ for all $%
v\in \mathcal{C}$ and $w\not\in W$. Further, every word in $\mathcal{B}$
must have an ancestor in $\mathcal{C}\times W$ and so $\mu
(B(x,e^{-n}R))\leq \sum_{vw\in \mathcal{C}\times W}m(M_{vw})$. Also, note
that $\mu (B(x,R))\geq \sum_{v\in \mathcal{C}}m(M_{v})$ since $%
M_{vu}\subseteq B(x,R)$ for all $v\in \mathcal{C}$ and $u\in \Lambda ^{\ast
} $. Hence, 
\begin{equation*}
\frac{\mu (B(x,R))}{\mu ((B(x,e^{-n}R))}\geq \frac{\sum_{v\in \mathcal{C}%
}m(M_{v})}{\sum_{vw\in \mathcal{C}\times W}m(M_{vw})}\geq \frac{\sum_{v\in 
\mathcal{C}}m(M_{v})}{(1-C_{3})^{k}\sum_{v\in \mathcal{C}}m(M_{v})}\geq
(1-C_{3})^{-k}
\end{equation*}%
and we can take $\gamma $ to be $1-(1-C_{3})^{k}$. Now $k$ is maximal and 
\begin{equation*}
C_{1}^{-k}<\frac{C_{1}^{2}C_{2}}{3}\frac{R}{e^{-n}R}=\frac{C_{1}^{2}C_{2}}{%
3e^{-n}}.
\end{equation*}%
So, 
\begin{equation*}
k\geq \frac{\log \left( C_{1}^{2}C_{2}/3\right) -\log e^{-n}}{\log (1/C_{1})}%
-1=\frac{n}{\log (1/C_{1})}+\frac{\log \left( -C_{1}^{2}C_{2}/3\right) }{%
\log (1/C_{1})}-1.
\end{equation*}%
We now apply Theorem \ref{thm:lowerlower} to get 
\begin{equation*}
\dimL\mu \geq \frac{\log (1-C_{3})^{-k}}{\log e^{n}}\geq \frac{-\log
(1-C_{3})}{n}\left( \frac{n}{\log (1/C_{1})}+\frac{\log \left(
C_{1}^{2}C_{2}/3\right) }{\log (1/C_{1})}-1\right) .
\end{equation*}%
Since $n$ was arbitrary, taking $n$ large gets the required bound on the
lower Assouad dimension and thus the measure is uniformly perfect.
\end{proof}

This result can be applied to a variety of measures. For instance, suppose
we are given an iterated function system (IFS) of similarities $%
\{S_{j}\}_{j=1}^{N}$ on $\mathbb{R}^{d}$ and probabilities $%
\{p_{j}\}_{j=1}^{N},$ with $p_{j}>0$ and $\sum_{i=1}^{N}p_{j}=1$. The
self-similar set associated with the IFS is the unique non-empty compact set 
$K$ such that $K=\bigcup_{j=1}^{N}S_{j}(K)$ which, without loss of
generality, can be assumed to be contained in $[0,1]^{d}$. We will further
assume that $K$ is not a singleton and thus perfect. The self-similar
measure $\mu $ is the unique probability measure satisfying 
\begin{equation*}
\mu =\sum_{j=1}^{N}p_{j}(\mu \circ S_{j}^{-1}).
\end{equation*}%
Given $v=(v_{j})_{j=1}^{n}\in \{1,\dots,N\}^{n},$ we let $%
S_{v}=S_{v_{1}}\circ S_{v_{2}}\circ \cdot \cdot \cdot \circ S_{v_{n}}$. If
we put $M_{v}=S_{v}([0,1]^{d})$, then the collection of sets $\{M_{v}\}$
satisfies the first two requirements of the Moran set construction above.
Condition \eqref{cond:notbunching} may not be satisfied, but by taking
iterates of the IFS it is eventually satisfied. If we also define the weight
function $m$ by $m(M_{vj})=p_{j}m(M_{v}),$ then the three conditions on the
weight function are also fulfilled. The self-similar measure is the measure $%
\mu $ arising from the weight function $m$ as in \eqref{meas}. Consequently,
applying Theorem \ref{thm:lowerlower} we obtain

\begin{corollary}
\label{thm:selfsimilarpos} The lower Assouad dimension of any non-degenerate
self-similar measure is positive.
\end{corollary}

\begin{remark}
This formula for the lower bound on the dimension is by no means sharp. For
an IFS that satisfies the strong separation condition and contraction
factors $r_{j},$ the approach gives $\log (1-\max p_{i})/\log (\min r_{i})$.
The same methods as used in \cite[Theorem 2.4]{FH} for the upper Assouad
dimension show that the actual value of the lower Assouad dimension is $\min
(\log p_{i}/\log r_{i}),$ the same as the minimal lower local dimension (see 
\cite{CM}).
\end{remark}

One can further extend Corollary~\ref{thm:selfsimilarpos} to equilibrium
Gibbs measures and quasi-Bernoulli measures on self-conformal sets in
general. A quasi-Bernoulli measure $\mu $ on the symbolic space $\{1,\dots
,N\}^{\N}$ is any probability measure that satisfies 
\begin{equation*}
c^{-1}\leq \frac{\mu ([v_{1},\dots ,v_{k}])}{\mu ([v_{1},\dots ,v_{l}])\mu
([v_{l+1},\dots ,v_{k}])}\leq c
\end{equation*}%
for all $v_{i}\in \{1,\dots ,N\}$ and $1\leq l\leq k$, where 
\begin{equation*}
\lbrack v_{1},\dots ,v_{k}]=\{w\in \{1,\dots ,N\}^{\N}\,:\,w_{i}=v_{i}\text{
for all }1\leq i\leq k\}
\end{equation*}%
and $c>0$ is a uniform constant.

Self-conformal sets satisfy the bounded distortion condition and expressing
them as such a Moran construction is straightforward, see e.g.~\cite%
{KaenmakiRossi15}. Similarly, the conditions on the mass functions are
easily seen to be satisfied.

\begin{corollary}
The lower Assouad dimension of the push-forward of a quasi-Bernoulli measure
onto non-degenerate self-conformal sets is positive.
\end{corollary}

\subsection{Self-affine measures}

The Moran construction detailed above is very flexible and also encompasses
self-affine measures. Showing this needs some extra work and our approach
here is similar to that of Xie, Jin, and Sun~\cite{Xie03} who proved that
self-affine sets are uniformly perfect. The approach relies chiefly on the
following easy lemma that only uses basic linear algebra. This lemma appears
in a slightly different form as Lemma 2.1 in \cite{Xie03}, but for
self-containment we have chosen to include its proof.

\begin{lemma}
\label{thm:matrixBound} Let $E=\{e_{1},\dots ,e_{d}\}$ be an orthonormal
basis of $\R^{d}$, let ${A},{B}$ be $d\times d$ matrices of which ${A}$ is
invertible. Then there exists a constant $\alpha _{{A}}>0$ depending only on 
${A}$ and $d$ such that 
\begin{equation*}
\max_{e\in E}\{\lvert {B}{A}\,e\rvert \}\geq \alpha _{{A}}\lVert {B}\rVert ,
\end{equation*}
where $\lVert B\rVert$ denotes the operator norm of $B$ acting as a linear
transformation on $\R^d$.
\end{lemma}

\begin{proof}
First note that there exists $x_{0}=\sum_{i=1}^{d}c_{i}e_{i}$ for some
scalars $c_{i}$ with $\sum_{i}\left\vert c_{i}\right\vert =1,$ such that $%
\lVert {B}{A}\rVert =\lvert {B}{A}x_{0}|$. By linearity, 
\begin{equation}
\lVert {B}{A}\rVert =\lvert {B}{A}x_{0}\rvert =\lvert c_{1}{B}{A}e_{1}+\dots
+c_{d}{B}{A}e_{d}\rvert \leq d\max_{1\leq i\leq d}\lvert {B}{A}e_{i}\rvert .
\end{equation}%
Thus the submuliplicativity of the matrix norm $\lVert .\rVert $ gives%
\begin{equation*}
\max_{e\in E}\{\lvert {B}{A}\,e\rvert \}\geq d^{-1}\lVert {B}{A}\rVert =%
\frac{\lVert {B}{A}\rVert \lVert {A}^{-1}\rVert }{d\lVert {A}^{-1}\rVert }%
\geq \frac{\lVert {B}\rVert }{d\lVert {A}^{-1}\rVert }.
\end{equation*}%
Letting $\alpha _{{A}}=(d\lVert {A}^{-1}\rVert )^{-1}$ finishes the proof.
\end{proof}

Let $f_{i}(x)={A}_{i}x+t_{i}$, $i=1,\dots,N,$ be affine maps such that ${A}%
_{i} $ is non-singular and $\lVert {A}_{i}\rVert <1$ for all $i$. The
self-affine set associated with $\{f_{i}\}$ is the unique compact set $K$
that satisfies 
\begin{equation*}
K=\bigcup_{i=1}^{N}f_{i}(K).
\end{equation*}%
We will assume that the attractor is not a singleton, which amounts to at
least two $f_{i},f_{j}$ having distinct fixed points, and prove

\begin{theorem}
Let $\mu $ be the push-forward of a quasi-Bernoulli measure on the
self-affine set $F$. If $F$ is not a singleton, the measure $\mu $ has
positive lower dimension and thus is uniformly perfect.
\end{theorem}

\begin{proof}
Without loss of generality we can assume that $K$ is not contained in any
proper subspace of $\mathbb{R}^{d}$, redefining the affine maps with
projections otherwise. So there exist $y_{1},\dots ,y_{d}\in K$ such that $%
\{y_{1},\dots ,y_{d}\}$ is linearly independent. Let ${Y}$ be the linear
transformation that maps $e_{i}$ onto $y_{i}$. As this is a change of basis, 
${Y}$ must be invertible.

Let $B_{R}=B(0,R)$ be the closed ball of radius $R$ and choose $R$ large
enough such that $f_{i}(B_{R})\subset B_{R}$ for all $i$. Then, $%
f_{j}(f_{i}(B_{R}))\subset f_{j}(B_{R})$, and generally $f_{vw}(B_{R})%
\subset f_{v}(B_{R})$ for all non-empty $v,w\in \{1,\dots ,N\}^{\ast }$.
Observe that the composition of affine maps is itself affine and $\diam %
f_{v_{1}\dots v_{k}}(B_{1})=2\lVert {A}\rVert $, where ${A}={A_{v_{1}}}\dots 
{A}_{v_{k}}$ is the linear component of $f_{v_{1}\dots v_{k}}$. Let $K$ be
the minimal integer such that $(\max_{i}\{\lVert {A}_{i}\rVert
\})^{K}<\alpha _{{Y}}/(6R)$ where $\alpha _{{Y}}$ is as in Lemma \ref%
{thm:matrixBound}. Let $\Lambda =\{1,\dots ,N\}^{K}$ and define $%
M_{\emptyset }=B_{R}$ and $M_{v}=f_{v}(B_{R})$, bearing in mind that a word $%
v\in \Lambda ^{k}$ is of length $k\cdot K$. This definition clearly
satisfies \eqref{cond:subset} and \eqref{cond:decr} in the Moran
construction definition.

For \eqref{cond:notsmall} we note that for all $v\in \Lambda ^{k}$ and $j\in
\Lambda $, 
\begin{align*}
\diam M_{vj}& =\diam({A}_{v_{1}}\cdots {A}_{v_{(kK)}}{A}_{j_{1}}\cdots {A}%
_{j_{K}}(B_{R}))=2R\left\Vert {A}_{v_{1}}\cdots {A}_{v_{(kK)}}{A}%
_{j_{1}}\cdots {A}_{j_{K}}\right\Vert \\
& \geq 2R\left\Vert {A}_{v_{1}}\cdots {A}_{v_{(kK)}}\right\Vert \left\Vert ({%
A}_{j_{1}}\cdots {A}_{j_{K}})^{-1}\right\Vert ^{-1} \\
& =\left\Vert ({A}_{j_{1}}\cdots {A}_{j_{K}})^{-1}\right\Vert ^{-1}\diam({A}%
_{v_{1}}\cdots {A}_{v_{(kK)}}(B_{R})) \\
& =\left\Vert ({A}_{j_{1}}\cdots {A}_{j_{K}})^{-1}\right\Vert ^{-1}\diam %
M_{v}.
\end{align*}%
Since $\Lambda $ is finite and all ${A}_{i}$ are invertible, there exists a
constant 
\begin{equation*}
C_{1}=\min_{v\in \Lambda }\left\Vert ({A}_{v_{1}}\cdots {A}%
_{v_{K}})^{-1}\right\Vert ^{-1}>0
\end{equation*}%
such that \eqref{cond:notsmall} is satisfied.

Finally we check \eqref{cond:notbunching}. Let $v\in \Lambda ^{k}$ and
recall that ${Y}$ maps the basis $E$ onto a linearly independent set of
points in $K$. Using Lemma \ref{thm:matrixBound} we obtain, 
\begin{equation*}
\max_{e\in E}\{\lvert {A}_{v_{1}}\cdots {A}_{v_{(kK)}}{Y}\,e\rvert \}\geq
\alpha _{{Y}}\left\Vert {A}_{v_{1}}\cdots {A}_{v_{(kK)}}\right\Vert =\frac{%
\alpha _{{Y}}}{2R}\diam M_{v}.
\end{equation*}%
But then 
\begin{equation*}
\diam f_{v}({Y}E)\geq \frac{\alpha _{{Y}}}{2R}\diam M_{v}
\end{equation*}%
and as ${Y}e_{i}=y_{i}$, we have ${Y}E\subset F$ and $f_{v}({Y}E)\subset K$.
Thus there exist two points in $M_{v}\cap F$ that are at least $({\alpha _{{Y%
}}}/{(2R))}\diam M_{v}$ apart. Since these two points must be contained in $%
M_{vi}$ and $M_{vj}$, respectively, and $\diam M_{vi},\diam M_{vj}\leq ({%
\alpha _{{Y}}}/{(6R))}\diam M_{v}$ we must have $i\neq j$, and further 
\begin{equation*}
d(M_{vi},M_{vj})\geq ({\alpha _{{Y}}}/{(3R))}\diam M_{v}.
\end{equation*}%
Thus Condition \eqref{cond:notbunching} is satisfied.

Letting $m(M_v) = \mu([v])$ for all $v\in\Lambda^{*}$ gives the correct
measure on $M=F$. Checking the conditions on the weight function is
straightforward and left to the reader.
\end{proof}

\begin{remark}
It was observed by K\"{a}enm\"{a}ki and Lehrb\"{a}ck, see \cite[Lemma 3.1]%
{KL}, that any doubling measure supported on a uniformly perfect metric
space has positive lower dimension. Our results above show that there are
many measures with positive lower Assouad dimension that are, in general,
far from doubling.
\end{remark}

\subsection{Lower dimension of Bedford-McMullen carpets}

One example of a class of self-affine measures are the pushforward measures
given by a Bernoulli probability measure on Bedford-McMullen carpets. In
this subsection we compute the exact lower Assouad dimension of these
measures. The result is analogous to the upper Assouad dimension for sponges
given in \cite{FH} and due to its similarity we will only give a brief
sketch of its proof.

Let $2\leq m<n$ be integers and consider maps of the form $f_{i}(x)=Ax+t_{i}$%
, where $1\leq i\leq N,$ $A$ is the diagonal matrix $A=\diag(1/m,1/n)$ and $%
t_{i}=[a_{i}/m\;b_{i}/n]^\top$ for some integers $0\leq a_{i}<m$ and $0\leq
b_{i}<n$. The attractor of the IFS $\{f_{1},\dots ,f_{N}\}$ is known as a
Bedford-McMullen carpet. If there exists $\varepsilon >0$ such that all $%
f_{i}([-\varepsilon ,1+\varepsilon ]^{2})$ are pairwise disjoint, we say
that the iterated function system satisfies the very strong separation
condition.

Given $p_{i}>0$ such that $\sum p_{i}=1$, let $\mu $ be the pushforward
measure of the Bernoulli measure on $\{1,\dots ,N\}^{\N}$ under the IFS. The
lower Assouad dimension of this self-affine measure is characterized by
finding a minimising column. We write 
\begin{equation*}
p_{\col}(i)=\sum_{\substack{ j\in\{1,\dots,N\}  \\ a_j=a_i}}p_j
\end{equation*}
for the measure of the column containing $f_i([0,1])$, that is $p_{\col}(i)
= \mu( [a_i/m,(a_i+1)/m]\times [0,1])$.

\begin{theorem}
Let $\mu $ be the self-affine measure of Bedford-McMullen type with
associated probabilities $p_{i}$ and contractions $f_{i}$. If the very
strong separation condition holds, then 
\begin{equation}  \label{eq:dimFormula}
\dimL\mu =\min_{1\leq j\leq N}\frac{-\log p_{\col}(j)}{\log m}+\min_{1\leq
i\leq N}\frac{\log p_{col}(i)/p_i}{\log n}.
\end{equation}
Furthermore, $\underline{H}(t)=\dimL\mu$ for small enough $t$ and so $\dimqL%
\mu=\dimL\mu$.
\end{theorem}

\begin{proof}
The key idea to establishing this dimension result are \textquotedblleft
approximate squares\textquotedblright, see \cite{FH} for details.
Heuristically, an approximate square is a collection of words such that the
corresponding set has uniformly comparable base and height, i.e.\ is
`almost' a square. We will construct approximate squares below and check
that they give rise to the dimension formula. The details that allow us to
transition from nested approximate squares to balls are based on the very
strong separation condition and contained in \cite{FH}; we decided to omit
them for brevity.

Let $k_{1}(R)$ and $k_{2}(R)$ be the unique integers such that $%
m^{-k_{1}(R)}\leq R<m^{-k_{1}(R)+1}$ and $n^{-k_{2}(R)}\leq
R<n^{-k_{2}(R)+1} $. Due to the common diagonal structure of the linear part
of $f_{i}$, the image $f_{v}([0,1]^{2})$ will be a rectangle aligned with
the first and second coordinate. If $v$ has length $k_{2}(R)$, the rectangle 
$f_{v}([0,1]^{2})$ will have height in $(n^{-1}R,R]$. Similarly, for any
word of length $k_{1}(R)$, the corresponding rectangle will have base in $%
(m^{-1}R,R]$. Given $0<r<R<1$ let $v_{r}\in \Lambda ^{k_{2}(r)}$ and $%
w_{r}\in \Lambda ^{\ast }$ such that $v_{r}w_{r}\in \Lambda ^{k_{1}(r)}$ and
consider the set 
\begin{equation*}
Q_{r}=\bigcup_{w\in \Lambda ^{\ast }}\{f_{v_{r}w}([0,1]^{2})\,:\,v_{r}w\in
\Lambda ^{k_{1}(r)}\text{ and }a_{(v_{r}w)_{i}}=a_{(v_{r}w_{r})_{i}}\text{
for all }1\leq i\leq k_{1}(r)\},
\end{equation*}%
that is the set of all images of words that have $v_{r}$ as the ancestor
(whose rectangle has height comparable to $r$) such that each rectangle
associated with $v_{r}w$ has base comparable to $r$ and the horizontal
translations all agree so all $f_{v_{r}w}([0,1]^{2})$ align in the same
column as $f_{v_{r}w_{r}}([0,1]^{2})$. Therefore $Q_{r}$ must have height
and base comparable to $r$ and is a (generic) approximate square. Its parent
approximate square of size $R$ is denoted by $Q_{R}$ and is the set given by 
\begin{equation*}
Q_{R}=\bigcup_{w\in \Lambda ^{\ast }}\{f_{v_{R}w}([0,1]^{2})\,:\,v_{R}w\in
\Lambda ^{k_{1}(R)}\text{ and }a_{(v_{R}w)_{i}}=a_{(v_{r}w_{r})_{i}}\text{
for all }1\leq i\leq k_{1}(R)\},
\end{equation*}%
where $v_{R}\in \Lambda ^{k_{2}(R)}$ is the parent word of $v_{r}$.

As mentioned above, it is sufficient to check $\mu (Q_{R})/\mu (Q_{r})$ for
all arbitrary approximate squares of the above form. Their measures are 
\begin{equation*}
\mu
(Q_{R})=\prod_{i=1}^{k_{2}(R)}p_{(v_{r}w_{r})_{i}}%
\prod_{i=k_{2}(R)+1}^{k_{1}(R)}p_{\col}((v_{r}w_{r})_{i})
\end{equation*}%
and 
\begin{equation*}
\mu
(Q_{r})=\prod_{i=1}^{k_{2}(r)}p_{(v_{r}w_{r})_{i}}%
\prod_{i=k_{2}(r)+1}^{k_{1}(r)}p_{\col}((v_{r}w_{r})_{i}).
\end{equation*}%
Notice that by definition we must either have 
\begin{equation}
k_{2}(R)<k_{1}(R)<k_{2}(r)<k_{1}(r)\quad \text{ or }\quad
k_{2}(R)<k_{2}(r)<k_{1}(R)<k_{1}(r).  \label{eq:cases}
\end{equation}%
In the first case we get 
\begin{align*}
\frac{\mu (Q_{R})}{\mu (Q_{r})}& =\prod_{i=k_{2}(R)+1}^{k_{1}(R)}p_{\col%
}((v_{r}w_{r})_{i})/p_{(v_{r}w_{r})_{i}}\left(
\prod_{i=k_{1}(R)+1}^{k_{2}(r)}p_{(v_{r}w_{r})_{i}}%
\prod_{i=k_{2}(r)+1}^{k_{1}(r)}p_{\col}((v_{r}w_{r})_{i})\right) ^{-1} \\
& \geq \left( \min_{1\leq j\leq N}\frac{p_{\col}(j)}{p(j)}\right)
^{k_{1}(R)-k_{2}(R)-1}\left( \min_{1\leq j\leq N}p(j)^{-1}\right)
^{k_{2}(r)-k_{1}(R)-1}\left( \min_{1\leq j\leq N}p_{\col}(j)^{-1}\right)
^{k_{1}(r)-k_{2}(r)-1} \\
& \geq C\left( \min_{1\leq j\leq N}\frac{p_{\col}(j)}{p(j)}\right)
^{k_{1}(R)-k_{2}(R)}\left( \min_{1\leq j\leq N}p_{\col}(j)^{-1}\right)
^{k_{1}(r)-k_{2}(r)},
\end{align*}%
for some $C>0$. Note that $k_{1}(t)-k_{2}(t)\sim \log (1/t)$ and so the
lower bound increases to infinity as $R\rightarrow 0$ and $r\rightarrow 0$,
irrespective of $R/r$. On the other hand, in the second case, we obtain 
\begin{align*}
\frac{\mu (Q_{R})}{\mu (Q_{r})}& =\left( \prod_{i=k_{2}(R)+1}^{k_{1}(R)}p_{%
\col}((v_{r}w_{r})_{i})\right) \left(
\prod_{i=k_{2}(R)+1}^{k_{2}(r)}p_{(v_{r}w_{r})_{i}}%
\prod_{i=k_{2}(r)+1}^{k_{1}(r)}p_{\col}((v_{r}w_{r})_{i})\right) ^{-1} \\
& =\left( \prod_{i=k_{2}(R)+1}^{k_{2}(r)}p_{\col}((v_{r}w_{r})_{i})\right)
\left(
\prod_{i=k_{2}(R)+1}^{k_{2}(r)}p_{(v_{r}w_{r})_{i}}%
\prod_{i=k_{1}(R)+1}^{k_{1}(r)}p_{\col}((v_{r}w_{r})_{i})\right) ^{-1} \\
& =\prod_{i=k_{2}(R)+1}^{k_{2}(r)}p_{\col%
}((v_{r}w_{r})_{i})/p_{(v_{r}w_{r})_{i}}\prod_{i=k_{1}(R)+1}^{k_{1}(r)}p_{%
\col}((v_{r}w_{r})_{i})^{-1} \\
& \geq C\left( \min_{1\leq j\leq N}\frac{p_{\col}(j)}{p_{j}}\right)
^{k_{2}(r)-k_{2}(R)}\left( \min_{1\leq j\leq N}p_{\col}(j)^{-1}\right)
^{k_{1}(r)-k_{1}(R)} \\
& =C\left( \min_{1\leq j\leq N}\frac{p_{\col}(j)}{p_{j}}\right) ^{\log
(R/r)/\log n}\left( \min_{1\leq j\leq N}p_{\col}(j)^{-1}\right) ^{\log
(R/r)/\log m}=C\left( \frac{R}{r}\right) ^{s}
\end{align*}%
for some uniform $C>0$ and $s$ as in \eqref{eq:dimFormula}. This shows that $%
\dimL\mu \geq s$.

Lastly, the second behaviour in \eqref{eq:cases} occurs when $r>
R^{1+\delta} $, where $1+\delta = \log n / \log m$. Therefore there exists a
word such that this minimum is achieved and we obtain $\underline{H}(t)\leq
s $ and $\underline{H}(t)$ is constant for $0<t<\delta$.
\end{proof}

\section{The lower Assouad dimension for self-similar measures of finite
type \label{FT}}

\subsection{Finite type measures}

In this section, we will prove that for a class of self-similar measures on $%
\mathbb{R}$, called finite type, the lower Assouad dimension coincides with
the minimal lower local dimension of the measure (Theorem~\ref{thm:localdim}%
). Many interesting self-similar measures that fail the open set condition
are of finite type, such as Bernoulli convolutions with Pisot contractions.
We begin by explaining what is meant by finite type.

Assume we are given an IFS of similarities, $S_{j}(x)=r_{j}x+d_{j}:\mathbb{%
R\rightarrow R}$ for $j=1,\dots,N$, where $N\geq 2$ and $0<\left\vert
r_{j}\right\vert <1,$ and probabilities $\{p_{k}\}_{j=1}^{N}$. By rescaling
and translation, there is no loss in assuming the convex hull of the
self-similar set $K$ is $[0,1]$. We let $\mu $ denote the self-similar
measure, $\mu (E)=\sum_{j=1}^{N}p_{j}\mu (S_{j}^{-1}(E))$.

Given any integer $n$ and $v=(v_{j})_{j=1}^{n}\in \{1,\dots,N\}^{n},$ we let 
$v^{-}=(v_{1},\dots,v_{n-1})$, $r_{v}=\prod_{i=1}^{n}r_{v_{i}}$ and $%
p_{v}=\prod_{j=1}^{n}p_{v_{j}}$. Put%
\begin{equation*}
\lambda =\min_{j=1,\dots,N}\left\vert r_{j}\right\vert
\end{equation*}%
and 
\begin{equation*}
\Lambda _{n}=\{v\in \{1,\dots ,N\}^{\ast }:\left\vert r_{v}\right\vert \leq
\lambda ^{n}\text{ and }\left\vert r_{v^{-}}\right\vert >\lambda ^{n}\}.
\end{equation*}

The notion of finite type was introduced by Ngai and Wang in \cite{NW}. The
definition we will use is slightly less general, but is simpler and includes
all the examples in $\mathbb{R}$ that we are aware of.

\begin{definition}
Assume $\{S_{j}\}$ is an IFS of similarities. The words $v,w \in \Lambda
_{n} $ are said to be neighbours if $S_{v}(0,1)\cap S_{w}(0,1)\neq \emptyset 
$. Denote by $\mathcal{N}(v)$ the set of all neighbours of $v$. We say that $%
v\in \Lambda _{n}$ and $w\in \Lambda _{m}$ have the same neighbourhood type
if there is a map $f(x)=\pm \lambda ^{n-m}x+c$ such that 
\begin{equation*}
f\circ S_{v}=S_{w}\text{ and }\{f\circ S_{u}:u \in \mathcal{N}%
(v)\}=\{S_{t}:t \in \mathcal{N}(w)\}.
\end{equation*}%
The IFS is said to be of \textbf{finite type} if there are only finitely
many neighbourhood types. Any associated self-similar measure is also said
to be of \textbf{finite type}.
\end{definition}

It was shown in \cite{Ng} that an IFS of finite type satisfies the weak
separation condition, but not necessarily the open set condition. For
instance, the IFS given by $S_{j}(x)=\pm \rho ^{-n_{j}}x+b_{j}$ where $\rho $
is a Pisot number\footnote{%
A Pisot number is an algebraic number greater than one, all of whose Galois
conjugates are strictly less than one in modulus. An example is the golden
mean.}, $n_{j}\in \mathbb{N}$ and $b_{j}\in \mathbb{Q[\rho ]}$, was shown to
be of finite type in \cite[Theorem 2.9]{NW}, but fails the open set
condition. The Bernoulli convolutions with contraction factors that are
inverses of Pisot numbers are self-similar measures associated with an IFS
of this form. As integers are also Pisot numbers, the self-similar measures
coming from an IFS $\{S_{j}(x)=x/d+j(d-1)/d\}_{j=0}^{m-1}$, for integer $%
d\geq 3,$ such as $m$-fold convolutions of the uniform Cantor measure on the
Cantor set of ratio $1/d,$ are another class of finite type measures.

\begin{definition}
For each positive integer $n$, let $h_{1},\dots ,h_{s_{n}}$ be the
collection of elements of the set $\{S_{v}(0),S_{v}(1):v \in \Lambda _{n}\}$%
, listed in increasing order. Set 
\begin{equation*}
\mathcal{F}_{n}=\{[h_{j},h_{j+1}]:1\leq j\leq s_{n}-1\text{ and }%
(h_{j},h_{j+1})\cap K\neq \emptyset \}\text{.}
\end{equation*}%
Elements of $\mathcal{F}_{n}$ are known as the \textbf{net intervals of
level }$n$.
\end{definition}

For each $\Delta \in \mathcal{F}_{n}$, $n\geq 1$, there is a unique element $%
\widehat{\Delta }\in \mathcal{F}_{n-1}$ which contains $\Delta ,$ called the 
\textit{parent} (of \textit{child} $\Delta )$. Given $\Delta =[a,b]\in 
\mathcal{F}_{n}$, we denote the \textit{normalized length} of $\Delta $ by $%
\ell _{n}(\Delta )=\lambda ^{-n}(b-a)$. By the \textit{neighbour set} of $%
\Delta $ we mean the ordered tuple 
\begin{equation*}
V_{n}(\Delta )=((a_{1},L_{1}),(a_{2},L_{2}),\dots ,(a_{j},L_{j})),
\end{equation*}%
where for each $i$ there is some $v\in \Lambda _{n}$ such that $\lambda
^{-n}r_{v}=L_{i}$ and $\lambda ^{-n}(a-S_{v}(0))=a_{i}$. Suppose $\Delta \in 
\mathcal{F}_{n}$ has parent $\widehat{\Delta }$. If $\widehat{\Delta }$ has
multiple children with the same normalized length and neighbourhood set as $%
\Delta ,$ order them from left to right as $\Delta _{1},\Delta _{2},\dots
,\Delta _{T}$. Let $t_{n}(\Delta )\in \{1,\dots ,T\}$ be the integer $t$
such that $\Delta _{t}=\Delta$.

\begin{definition}
The \textbf{characteristic vector}\textit{\ of }$\Delta \in \mathcal{F}_{n}$
is defined to be the triple 
\begin{equation*}
\mathcal{C}_{n}(\Delta )=(\ell _{n}(\Delta ),V_{n}(\Delta ),t_{n}(\Delta )).
\end{equation*}
\end{definition}

A very important fact, shown in \cite[Theorem 2.7]{HHS}, is that an IFS\ of
finite type admits only finitely many characteristic vectors. The
characteristic vectors are of fundamental importance because, as we will
see, we can obtain key information about the local behaviour of any
associated self-similar measure from them.

By the \textbf{symbolic representation} of a net interval $\Delta \in 
\mathcal{F}_{n}$ we mean the $(n+1)$-tuple $(\mathcal{C}_{0}(\Delta
_{0}),\dots,\mathcal{C}_{n}(\Delta _{n})),$ where $\Delta _{0}=[0,1]$, $%
\Delta _{n}=\Delta $, and for each $j=1,\dots,n$, $\Delta _{j-1}$ is the
parent of $\Delta _{j}$. Similarly, for each $x\in K=\supp\mu $ the symbolic
representation of $x$ will be the sequence of characteristic vectors $[x]=(%
\mathcal{C}_{0}(\Delta _{0}),\mathcal{C}_{1}(\Delta _{1}),\dots)$ where $%
x\in \Delta _{n}\in \mathcal{F}_{n}$ for each $n$ and $\Delta _{j-1}$ is the
parent of $\Delta _{j}$. Conversely, every sequence of characteristic
vectors $(\gamma _{0},\gamma _{1},\dots)$ where $\gamma _{0}=\mathcal{C}%
_{0}(\Delta _{0})$ and $\gamma _{j}$ is the parent of $\gamma _{j+1}$ is the
symbolic representation of a unique $x\in K$. We will write $\Delta _{n}(x)$
for a net interval of level $n$ containing $x$.

By a \textbf{path} we mean a segment of a symbolic representation. A \textbf{%
loop class }is a set of characteristic vectors $\mathcal{L}$ with the
property that given any $\chi ,\psi \in \mathcal{L}$ there is some finite
path $\eta $ in $\mathcal{L}$ so that $(\chi ,\eta ,\psi )$ is a path (in $%
\mathcal{L}$).

\begin{definition}
Let $\Delta =[a,b]\in \mathcal{F}_{n}$ and let $\widehat{\Delta }=[c,d]\in 
\mathcal{F}_{n-1}$ denote its parent net interval. Assume $V_{n}(\Delta
)=((a_{1},L_{1}),\dots ,(a_{I},L_{I}))$ and $V_{n-1}(\widehat{\Delta }%
)=((c_{1},M_{1}),\dots ,(c_{J},M_{J}))$. The \textbf{primitive transition
matrix, }denoted 
\begin{equation*}
T(\mathcal{C}_{n-1}(\widehat{\Delta }),\mathcal{C}_{n}(\Delta )),
\end{equation*}%
is the $I\times J$ matrix whose $(i,j)^{\prime }th$ entry, $T_{ij},$ is
defined as follows: Put $T_{ij}=p_{\omega }$ if there exists $v\in \Lambda
_{n-1}$ with $v\omega \in \Lambda _{n}$, $S_{v\omega }(x)=\lambda
^{n}(L_{i}x-a_{i})+a$ and $S_{v}(x)=\lambda ^{n-1}(M_{j}x-c_{j})+c$. If
there is no such $\omega $, we put $T_{ij}=0$.
\end{definition}

Given a path $(\gamma _{J},\gamma _{J+1},\dots,\gamma _{N})$, we write $%
T(\gamma _{J},\gamma _{J+1},\dots,\gamma _{N})$ for the product 
\begin{equation*}
T(\gamma _{J},\gamma _{J+1},\dots,\gamma _{N})=T(\gamma _{J},\gamma
_{J+1})T(\gamma _{J+1},\gamma _{J+2})\cdot \cdot \cdot T(\gamma
_{N-1},\gamma _{N}).
\end{equation*}
For brevity we write $\left\Vert (T_{ij})\right\Vert =\sum_{i,j}\left\vert
T_{ij}\right\vert $ and note the following critical fact proven in \cite[\S %
3.2]{HHS}.

\begin{lemma}
\label{trans}There are constants $a,b>0$ such that whenever $\Delta _{n}$ is
a net interval of level $n$ with symbolic representation $(\gamma
_{0},\gamma _{1},\dots,\gamma _{n})$, then 
\begin{equation*}
a\left\Vert T(\gamma _{0},\gamma _{1},\dots,\gamma _{n})\right\Vert \leq \mu
(\Delta _{n})\leq b\left\Vert T(\gamma _{0},\gamma _{1},\dots,\gamma
_{n})\right\Vert .
\end{equation*}
\end{lemma}

This lemma is useful because the lower Assouad dimensions for self-similar
measures of finite type can be deduced from the knowledge of the measure of
net intervals, as we see next.

\begin{lemma}
\label{thm:bddNets} If $\mu $ is of finite type and $\dimL\mu \leq d$, then
for each $\varepsilon >0$ there are $x_{i}\in \supp\mu $ and net intervals $%
\Delta _{N_{i}}(x_{i})\supseteq \Delta _{n_{i}}(x_{i})$ with $%
n_{i}-N_{i}\rightarrow \infty $ such that 
\begin{equation}
\frac{\mu (\Delta _{N_{i}}(x_{i}))}{\mu (\Delta _{n_{i}}(x_{i}))}<\lambda
^{(d+\varepsilon )(N_i - n_{i})}\text{.}  \label{3}
\end{equation}
\end{lemma}

\begin{proof}
Suppose the statement above is false. As there are only finitely many
characteristic vectors, all normalized lengths of net intervals are
comparable. Thus we may choose $c>0$ so that $\diam(\Delta _{n})\geq
c\lambda ^{n}$ for all net intervals $\Delta _{n}$ of level $n$. Given any
net interval, $\Delta _{n},$ of level $n$, we will write $\Delta _{n}^{R}$
and $\Delta _{n}^{L}$ for the adjacent net intervals of level $n$ to the
right and left of $\Delta _{n}$ respectively, should these exist.

Let $x\in \supp\mu $ and consider $r<R$ where $R/r\rightarrow \infty $.
Choose $N,n$ such that $3\lambda ^{N}\leq R<3\lambda ^{N-1}$ and $c\lambda
^{n+1}<r\leq c\lambda ^{n}$. Then $n-N\rightarrow \infty $ and 
\begin{equation*}
\frac{\mu (B(x,R))}{\mu (B(x,r))}\geq \frac{\mu (B(x,3\lambda ^{N}))}{\mu
(B(x,c\lambda ^{n}))}
\end{equation*}

As all net intervals of level $n$ have diameter between $c\lambda ^{n}$ and $%
\lambda ^{n}$, for any $x\in \supp\mu $ we have%
\begin{equation*}
B(x,3\lambda ^{N})\cap \supp\mu \supseteq (\Delta _{N}(x)\cup \Delta
_{N}^{R}\cup \Delta _{N}^{L})\cap \supp\mu
\end{equation*}%
and%
\begin{equation*}
B(x,c\lambda ^{n})\cap \supp\mu \subseteq (\Delta _{n}(x)\cup \Delta
_{n}^{R}\cup \Delta _{n}^{L})\cap \supp\mu \text{.}
\end{equation*}%
Thus 
\begin{equation*}
\mu (B(x,3\lambda ^{N}))\geq \max \{\mu (\Delta _{N}(x)),\mu (\Delta
_{N}^{R}),\mu (\Delta _{N}^{L})\},
\end{equation*}%
while%
\begin{equation*}
\mu (B(x,c\lambda ^{n}))\leq 3\max \{\mu (\Delta _{n}(x)),\mu (\Delta
_{n}^{R}),\mu (\Delta _{n}^{L})\}\text{.}
\end{equation*}

First, suppose $\mu (B(x,c\lambda ^{n}))\leq 3\mu (\Delta _{n}(x))$. Since
we are assuming \eqref{3} fails, 
\begin{equation*}
\frac{\mu (B(x,3\lambda ^{N}))}{\mu (B(x,c\lambda ^{n}))}\geq \frac{\mu
(\Delta _{N}(x))}{3\mu (\Delta _{n}(x))}\geq \frac{1}{3}\lambda
^{(d+\varepsilon )(N-n)}\geq C\left( \frac{R}{r}\right) ^{d+\varepsilon }
\end{equation*}
for a suitable constant $C,$ independent of $x,R,r$.

Otherwise, without loss of generality, $\mu (B(x,c\lambda ^{n}))\leq 3\mu
(\Delta _{n}^{L})$. Notice that $\Delta _{n}^{L}$ is either a child of $%
\Delta _{N}(x)$ or $\Delta _{N}^{L}$. If $\Delta _{n}^{L}\subseteq \Delta
_{N}(x)$ and we let $y\in \Delta _{n}^{L}\cap \supp\mu $, then $\Delta
_{N}(x)=\Delta _{N}(y)$ and $\Delta _{n}^{L}=\Delta _{n}(y),$ so we have 
\begin{equation*}
\frac{\mu (B(x,3\lambda ^{N}))}{\mu (B(x,c\lambda ^{n}))}\geq \frac{\mu
(\Delta _{N}(y))}{3\mu (\Delta _{n}(y))}\geq \frac{1}{3}\lambda
^{(d+\varepsilon )(N-n)}\geq C\left( \frac{R}{r}\right) ^{d+\varepsilon }.
\end{equation*}%
If, instead $\Delta _{n}^{L}\subseteq \Delta _{N}^{L}$ the arguments are
similar, just take $y$ to be the right endpoint of $\Delta _{n}^{L}$ and
then $\Delta _{N}^{L}=\Delta _{N}(y)$ and $\Delta _{n}^{L}=\Delta _{n}(y)$.

Consequently, 
\begin{equation*}
\frac{\mu (B(x,R))}{\mu (B(x,r))}\geq C\left( \frac{R}{r}\right)
^{d+\varepsilon }
\end{equation*}%
for all $x\in \supp\mu $ and $r<R$ with $R/r\rightarrow \infty $ and that
implies $\dimL\mu \geq d+\varepsilon $; a contradiction.
\end{proof}

\subsection{Lower Assouad dimension for measures of finite type}

We are now ready to prove the main result of this section.

\begin{theorem}
\label{thm:localdim} If $\mu $ is any self-similar measure of finite type,
then%
\begin{equation*}
\dimL\mu =\inf \{\underline{\dim }_{\loc}\mu (x):x\in \supp\mu \}.
\end{equation*}
\end{theorem}

\begin{proof}
Throughout the proof $C$ will denote a positive constant that may change
from one occurrence to another. Let $d=\inf_{x}\{\underline{\dim }_{\loc}\mu
(x)\}$ and assume for a contradiction that $\dimL\mu <d,$ say $\dimL\mu
<d-3\varepsilon $ for $\varepsilon >0$. We will show that this implies the
existence of points which have local dimension strictly less than $d$.

By Lemma~\ref{thm:bddNets}, there are $x_{i}\in \supp\mu $ and $N_{i}<n_{i}$
such that $n_{i}-N_{i}\rightarrow \infty $ and 
\begin{equation*}
\frac{\mu (\Delta _{N_{i}}(x_{i}))}{\mu (\Delta _{n_{i}}(x_{i}))}<\lambda
^{(d-2\varepsilon )(N_i-n_{i})}\text{ for all }i\text{.}
\end{equation*}%
It follows from Lemma \ref{trans} that there is a constant $C$ such that if $%
\Delta _{n_{i}}(x_{i})$ has symbolic representation $(\gamma _{0},\gamma
_{1}^{(i)},\dots,\gamma _{n_{i}}^{(i)})$, then%
\begin{eqnarray*}
C\lambda ^{(d-2\varepsilon )(N_i - n_{i})} &\geq &C\frac{\mu (\Delta
_{N_{i}}(x_{i}))}{\mu (\Delta _{n_{i}}(x_{i}))}\geq \frac{\left\Vert
T(\gamma _{0},\gamma _{1}^{(i)},\dots,\gamma _{N_{i}}^{(i)})\right\Vert }{%
\left\Vert T(\gamma _{0},\gamma _{1}^{(i)},\dots,\gamma
_{n_{i}}^{(i)})\right\Vert } \\
&\geq &\frac{\left\Vert T(\gamma _{0},\gamma _{1}^{(i)},\dots,\gamma
_{N_{i}}^{(i)})\right\Vert }{\left\Vert T(\gamma _{0},\gamma
_{1}^{(i)},\dots,\gamma _{N_{i}}^{(i)})\right\Vert \left\Vert T(\gamma
_{N_{i}}^{(i)},\dots,\gamma _{n_{i}}^{(i)})\right\Vert }.
\end{eqnarray*}%
Thus%
\begin{equation}
\left\Vert T(\gamma _{N_{i}}^{(i)},\dots,\gamma _{n_{i}}^{(i)})\right\Vert
\geq C\lambda ^{(d-2\varepsilon )(n_{i}-N_{i})}.  \label{1.2}
\end{equation}

The path, $(\gamma _{N_{i}}^{(i)},\dots,\gamma _{n_{i}}^{(i)})$, can be
rewritten as $(\chi _{0}^{(i)},\sigma _{1}^{(i)},\chi
_{1}^{(i)},\dots,\sigma _{k_{i}}^{(i)})$, where for each $j\geq 1,\sigma
_{j}^{(i)}$ is a path in a distinct maximal loop class $L_{j}^{(i)},$ $\chi
_{j}^{(i)}$ is a minimal length path joining the last letter of $\sigma
_{j}^{(i)}$ (a characteristic vector in $L_{j}^{i}$) to the first letter of $%
\sigma _{j+1}^{(i)}$ (a characteristic vector in $L_{j+1}^{i}$), and $\chi
_{0}^{(i)}$ is a path from the first letter of $\gamma _{N_{i}}^{(i)}$ to
the first letter of $\sigma _{1}^{(i)}$.

The finite type property ensures that there are only finitely many maximal
loop classes and only finitely many characteristic vectors in each loop
class. Hence there can only be finitely many of these minimal joining paths $%
\chi _{j}^{(i)}$ over all $i,j$. Thus $\sup_{i,j}\Vert T(\chi
_{j}^{(i)})\Vert $ is bounded and $\sup_{i}\{\sup_{j}$length$(\chi
_{j}^{(i)})\}\leq \sup_{i}A^{(i)}<\infty $. Since it is not possible to
return to a maximal loop class after leaving it, the numbers $k_{i}$ are
bounded, say by $k$. Hence there is a constant $C$ such that 
\begin{equation}
\left\Vert T(\gamma _{N_{i}}^{(i)},\dots,\gamma _{n_{i}}^{(i)})\right\Vert
\leq \prod\limits_{j=0}^{k_{i}-1}\left\Vert T(\chi _{j}^{(i)})\right\Vert
\prod\limits_{j=1}^{k_{i}}\left\Vert T(\sigma_{j}^{(i)})\right\Vert \leq
C^{k}\prod\limits_{j=1}^{k_{i}}\left\Vert T(\sigma_{j}^{(i)})\right\Vert .
\label{1}
\end{equation}

Let $l_{j}^{(i)}$ denote the length of the path $\sigma _{j}^{(i)}$. Then 
\begin{equation}
\sum_{j=1}^{k_{i}}l_{j}^{(i)}\leq
n_{i}-N_{i}=\sum_{j=1}^{k_{i}}l_{j}^{(i)}+\sum_{j=0}^{k_{i-1}}\text{length}%
(\chi _{j}^{(i)})\leq \sum_{j=1}^{k_{i}}l_{j}^{(i)}+kA^{(i)},  \label{1.1}
\end{equation}%
so $\sum_{j=1}^{k_{i}}l_{j}^{(i)}\rightarrow \infty $ as $i\rightarrow
\infty $. Putting together these observations we see that for large enough $%
n_{i}-N_{i},$ \eqref{1.2} gives%
\begin{eqnarray*}
\frac{\log \left\Vert T(\gamma _{N_{i}}^{(i)},\dots,\gamma
_{n_{i}}^{(i)})\right\Vert }{n_{i}-N_{i}} &\geq &\frac{(d-2\varepsilon
)(n_{i}-N_{i})\log \lambda +\log C}{n_{i}-N_{i}} \\
&\geq &(d-2\varepsilon )\log \lambda -\frac{\varepsilon }{2}|\log \lambda |,
\end{eqnarray*}%
while (\ref{1}-\ref{1.1}) imply 
\begin{eqnarray*}
\frac{\log \left\Vert T(\gamma _{N_{i}}^{(i)},\dots,\gamma
_{n_{i}}^{(i)})\right\Vert }{n_{i}-N_{i}} &\leq &\frac{\log C^{k}+\log
\prod\limits_{j=1}^{k_{i}}\left\Vert T(\sigma _{j}^{(i)})\right\Vert }{%
\sum_{j=1}^{k_{i}}l_{j}^{(i)}} \\
&\leq &\frac{\sum_{j=1}^{k_{i}}\log \left\Vert T(\sigma
_{j}^{(i)})\right\Vert }{\sum_{j=1}^{k_{i}}l_{j}^{(i)}}+\frac{\varepsilon }{2%
}|\log \lambda |.
\end{eqnarray*}%
Hence%
\begin{equation}
\frac{\sum_{j=1}^{k_{i}}\log \left\Vert T(\sigma _{j}^{(i)})\right\Vert }{%
\sum_{j=1}^{k_{i}}l_{j}^{(i)}}\geq (d-\varepsilon )\log \lambda \text{ for
large }i\text{.}
\end{equation}%
and that implies that $\log \left\Vert T(\sigma _{j}^{(i)})\right\Vert \geq
(d-\varepsilon )l_{j}^{(i)}$ for some $j=j_{i}$. There is no loss of
generality in assuming $j_{i}=1$. Thus 
\begin{equation}
\left\Vert T(\sigma _{1}^{(i)})\right\Vert \geq \lambda
^{l_{1}^{(i)}(d-\varepsilon )}.  \label{1.4}
\end{equation}

We will now construct $x\in \supp\mu $ with $\underline{\dim }_{\loc}\mu
(x)\leq d-\varepsilon /2$ by constructing a symbolic representation from the
symbolic representations of a suitable subsequence of the $(x_{i})$. We will
rely on the fact that there will be a subsequence of (the symbolic
representations for) $x_{i}$ and index $j_{i}$ such that all $\sigma
_{j_{i}}^{(i)}$ belong to the same loop class and their lengths are
unbounded in $i$.

As there are only finitely many maximal loop classes, there must be some
subsequence such that all $\sigma_{1}^{(i)}$ (for $i$ in the subsequence)
belong to the same maximal loop class.

Suppose $\sup_{i}\ell _{1}^{(i)}<\infty $. As there are finitely many
characteristic vectors, there can be only finitely many paths of length at
most $\sup_{i}\ell _{1}^{(i)}$ and hence $\sup_{i}\left\Vert T(\sigma
_{1}^{(i)})\right\Vert <\infty $. Consider again inequality \eqref{1} with
this additional information 
\begin{equation*}
\left\Vert T(\gamma _{N_{i}}^{(i)},\dots,\gamma _{n_{i}}^{(i)})\right\Vert
\leq C^{k}\prod\limits_{j=1}^{k_{i}}\left\Vert T(\sigma
_{j}^{(i)})\right\Vert \leq C^{k+1}\prod\limits_{j=2}^{k_{i}}\left\Vert
T(\sigma _{j}^{(i)})\right\Vert .
\end{equation*}%
Since 
\begin{equation*}
\sum_{j=2}^{k_{i}}l_{j}^{(i)}\leq n_{i}-N_{i}=\ell
_{1}^{(i)}+\sum_{j=2}^{k_{i}}l_{j}^{(i)}+\sum_{j=0}^{k_{i-1}}\text{length}%
(\chi _{j}^{(i)})\leq \sum_{j=2}^{k_{i}}l_{j}^{(i)}+\ell _{1}^{(i)}+kA^{(i)}
\end{equation*}%
and $kA^{(i)}+\ell _{1}^{(i)}$ is bounded over $i$, the same reasoning as
used to deduce \eqref{1.4} shows that for some further subsequence and index 
$j_{i}\in \{2,\dots,k\},$ which we can assume without loss of generality is $%
2$, we have 
\begin{equation*}
\left\Vert T(\sigma _{2}^{(i)})\right\Vert \geq \lambda
^{l_{2}^{(i)}(d-\varepsilon )}
\end{equation*}%
with all $\sigma _{2}^{(i)}$ belonging to the same maximal loop class.

If $\sup_{i}\ell _{2}^{(i)}<\infty $, we repeat the argument. As there are
only finitely many maximal loop classes, we must eventually find a
subsequence of the indices $i$ and index $j$ such that the paths $\sigma
_{j}^{(i)}=\rho _{i},$ all are in the same maximal loop class $\Lambda $,
their lengths $\ell ^{(i)}=\ell _{j}^{(i)}\rightarrow \infty $ as $%
i\rightarrow \infty $ and 
\begin{equation*}
\left\Vert T(\rho _{i})\right\Vert \geq \lambda ^{\ell
_{j}^{(i)}(d-\varepsilon )}.
\end{equation*}

We will now `stitch' these paths together to obtain the $x\in \supp\mu $
required for the contradiction. For each pair of characteristic vectors, $%
\chi ,\psi ,$ in $\Lambda $, choose a path $\eta _{\chi ,\psi }$ (in $%
\Lambda $) with first letter $\chi $ and last letter $\psi $. Choose, also,
a path $\eta _{\psi }$ from $\gamma _{0}$ to each $\psi \in \Lambda $. Let $%
\mathcal{S}$ denote the finite set consisting of the chosen paths $\eta
_{\chi ,\psi }$, $\eta _{\psi }$. Since a transition matrix contains a
non-zero entry in each column, there is some constant $c_{0}>0$ such that $%
\left\Vert T(\eta ,\sigma )\right\Vert \geq c_{0}\left\Vert T(\sigma
)\right\Vert $ for all $\eta \in \mathcal{S}$ and all admissible paths $%
\sigma $ (meaning, $(\eta ,\sigma )$ is a path). Choose $i_{1}$ such that 
\begin{equation*}
\frac{\left\vert \log c_{0}\right\vert }{\ell ^{(i_{i})}}<\frac{\varepsilon
|\log \lambda |}{2}
\end{equation*}%
and select a path $\nu _{1}\in \mathcal{S}$ joining $\gamma _{0}$ to the
path $\rho _{i_{1}}$.

Next, as $\nu _{1},\rho _{i_{1}}$ are fixed and $\mathcal{S}$ is finite, we
can choose $c_{1}>0$ such that 
\begin{equation*}
\left\Vert T(\nu _{1},\rho _{i_{1}},\eta ,\sigma )\right\Vert \geq
c_{1}\left\Vert T(\sigma )\right\Vert
\end{equation*}%
for all admissible paths $\eta \in \mathcal{S}$ and $\sigma $. Then choose $%
i_{2}>i_{1}$ such that 
\begin{equation*}
\frac{\left\vert \log c_{1}\right\vert }{\ell ^{(i_{2})}}<\frac{\varepsilon
|\log \lambda |}{2}.
\end{equation*}%
As $\rho _{i_{1}}$ and $\rho _{i_{2}}$ belong to the same maximal loop class 
$\Lambda $, there is some path $\nu _{2}$ joining the last letter of $\rho
_{i_{1}}$ to the first letter of $\rho _{i_{2}}$. Having found such a path,
choose $c_{2}>0$ so 
\begin{equation}
\left\Vert T(\nu _{1},\rho _{i_{1}},\nu _{2},\rho _{i_{2}},\eta ,\sigma
)\right\Vert \geq c_{2}\left\Vert T(\sigma )\right\Vert  \label{1.5}
\end{equation}%
for all admissible paths $\eta \in \mathcal{S}$ and $\sigma $. Repeat this
procedure to construct $\nu _{j},\rho _{i_{j}},j=1,2,\dots$ and then let $x$
be the element of $\supp\mu $ with symbolic representation%
\begin{equation*}
\lbrack x]=(\nu _{1},\rho _{i_{1}},\nu _{2},\rho _{i_{2}},\dots).
\end{equation*}

It only remains to verify that $\underline{\dim }_{\loc}\mu (x)\leq
d-\varepsilon /2$. Towards this, let $\cM_{n}(x)=\mu (\Delta _{n}(x))+\mu
(\Delta _{n}^{R})+\mu (\Delta _{n}^{L})$. As was essentially observed in 
\cite[Theorem 2.6]{HHN}, 
\begin{equation*}
\underline{\dim }_{\loc}\mu (x)=\liminf_{n}\frac{\log \cM_{n}(x)}{n\log
\lambda }.
\end{equation*}%
If $n=\sum_{j=1}^{J}($length$(\nu _{j})+\ell ^{(i_{j})})$, then $\Delta
_{n}(x)=(\nu _{1},\rho _{i_{1}},\nu _{2},\rho _{i_{2}},\dots,\rho _{i_{J}})$%
. Thus \eqref{1.5} yields 
\begin{equation*}
\cM_{n}(x)\geq \mu (\Delta _{n}(x))\geq C\left\Vert T(\nu _{1},\rho
_{i_{1}},\nu _{2},\rho _{i_{2}},\dots,\rho _{i_{J}})\right\Vert \geq
Cc_{J-1}\left\Vert T(\rho _{i_{J}})\right\Vert ,
\end{equation*}%
and so%
\begin{equation*}
\frac{\log \cM_{n}(x)}{n\log \lambda }\leq \frac{\log C+\log c_{J-1}+\log
\left\Vert T(\rho _{i_{J}})\right\Vert }{n\log \lambda }.
\end{equation*}%
Recall that $\left\Vert T(\rho _{i})\right\Vert \geq \lambda ^{\ell
^{(i)}(d-\varepsilon )}$ , hence as $n\geq \ell ^{(i_{J})}$%
\begin{equation*}
\frac{\log \left\Vert T(\rho _{i_{J}})\right\Vert }{n\log \lambda }\leq
(d-\varepsilon ).
\end{equation*}%
Furthermore, the choice of $i_{J}$ ensures that 
\begin{equation*}
\left\vert \frac{\log c_{J-1}}{n\log \lambda }\right\vert \leq \left\vert 
\frac{\log c_{J-1}}{\ell ^{(i_{J})}\log \lambda }\right\vert <\frac{%
\varepsilon }{2}.
\end{equation*}%
Consequently, for large enough $n$ of this form, 
\begin{equation*}
\frac{\log \cM_{n}(x)}{n\log \lambda }\leq \frac{\varepsilon }{4}+\frac{%
\varepsilon }{2}+d-\varepsilon \leq d-\frac{\varepsilon }{2}.
\end{equation*}%
That proves $\underline{\dim }_{\loc}\mu (x)\leq d-\varepsilon /2$ as we
claimed, contradicting the initial assumption of the proof that $d=\inf_{x}\{%
\underline{\dim }_{\loc}\mu (x)\}.$
\end{proof}

\begin{remark}
Although we know from Example \ref{qLneLoc} that this result is not true for
all measures, it would be interesting to know if was true for all
self-similar measures.
\end{remark}

\section{$L^{p}$-improving results\label{LI}}

A measure $\mu $ on $[0,1]^{d}$ is said to be $L^{p}$-improving if there is
some $p>2$ so that $\mu \ast f\in L^{p}$ whenever $f\in L^{2}$. An
application of the open mapping theorem implies that in this case there is a
constant $C$ such that $\left\Vert \mu \ast f\right\Vert _{p}\leq
C\left\Vert f\right\Vert _{2}$ for all $f\in L^{2}$. The Hausdorff-Young
inequality shows that any measure $\mu $ whose Fourier transform $\widehat{%
\mu }\in \ell ^{q}$ for some $q<\infty $ is $L^{p}$-improving. The uniform
Cantor measures on Cantor sets with ratios of dissection bounded away from
zero are also $L^{p}$-improving (but their transforms need not tend to zero) 
\cite{Ch}. Conversely, a point mass measure is not $L^{p}$-improving since
it acts as an isometry on the $L^{p}$ spaces.

It is known that if $\mu $ is $L^{p}$-improving, then the Hausdorff and
energy dimensions of $\mu $ are positive \cite{HR}. It is natural to ask if
a similar statement can be made about the lower Assouad dimension of $\mu $.
In this section, we will show that while it is true that $\inf_{x}\{%
\underline{\dim }_{\loc}\mu (x)\}>0$ for an $L^{p}$-improving measure, it is
not necessary for $\dimqL\mu >0,$ or even for $\dimqL\supp\mu >0$.

\begin{proposition}
If $\mu :L^{2}([0,1]^{d})\rightarrow L^{p}([0,1]^{d})$ for $p>2$, then $%
\underline{\dim }_{\loc}\mu (x)\geq d(\frac{1}{2}-\frac{1}{p})$ for every $%
x\in \supp\mu $.
\end{proposition}

\begin{proof}
Suppose this is not true, say $\underline{\dim }_{\loc}\mu (x)=\varepsilon $
for some $\varepsilon <d(1/2-1/p)$. Then for any $\delta >0$ there are $%
r_{n}\rightarrow 0$ such that $\mu (B(x,r_{n}))\geq r_{n}^{\varepsilon
+\delta }$. Let $f_{n}=1_{B(x,2r_{n})}$, so $\left\Vert f_{n}\right\Vert
_{2}\sim \sqrt{r_{n}^{d}}$. Note that if $z\in B(0,r_{n})$ and $t\in
B(x,r_{n}),$ then $z-t\in B(x,2r_{n})$ so $\mu \ast f_{n}(z)\geq \mu
(B(x,r_{n}))$. Hence for some constants $C_{1},C_{2},C_{3}$ (independent of $%
n)$ and all $r_{n}$, 
\begin{equation*}
C_{1}r_{n}^{d/2}\geq C_{2}\left\Vert f_{n}\right\Vert _{2}\geq \left\Vert
\mu \ast f_{n}\right\Vert _{p}\geq \mu (B(x,r_{n}))m(B(0,r_{n}))^{1/p}\geq
C_{3}r_{n}^{\varepsilon +\delta }r_{n}^{d/p}\text{.}
\end{equation*}%
But this is impossible as $\varepsilon +d/p+\delta <d/2$ for small $\delta
>0 $.
\end{proof}

We will give two examples to see this does not extend to the quasi-lower
Assouad dimension.

\begin{example}
A set $E\subseteq \lbrack 0,1]$ with $\dimqL E=0$ and a measure $\mu $
supported on $E$ that is $L^{p}$-improving: In \cite[Theorem 2]{Sa}, Salem
proved that the Fourier transform of the uniform Cantor measure supported on
suitable random Cantor sets is almost surely in $\ell ^{p}$ for some $%
p<\infty $. Such a measure is $L^{p}$-improving. We will show that we can
construct a suitable Cantor set so that its quasi-lower Assouad dimension is
zero.

We will follow the notation of Salem's paper. To begin, choose a rapidly
growing sequence $\{n_{j}\}$ and put $m_{j}\sim n_{j}\log 3/\log \log n_{j}$%
. For $k\in \Lambda _{j}=\{n_{j},\dots,n_{j}+m_{j}\},$ $j=1,2,\dots$, put $%
a_{k}=1/\log k$, $b_{k}=2/\log k,$ and otherwise put $a_{k}=\frac{1}{3}+%
\frac{1}{\log k}$, $b_{k}=\frac{1}{3}+\frac{2}{\log k}$. The sequence $%
\{n_{j}\}$ should be sufficiently sparse that $n_{j+1}\gg n_{j}+m_{j}$ and $%
\prod\limits_{k=1}^{n_{j}}a_{k}\geq 3^{-n_{j}}$. Now construct a random
Cantor set with ratio of dissection at step $k$ equal to $\xi _{k}$ where $%
\xi _{k}$ is chosen uniformly over the interval $[a_{k},b_{k}]$. We have $%
b_{k}-a_{k}=1/\log k$ and $(\log \log k)/k\rightarrow 0$. Furthermore, it is
easy to see that $\liminf (a_{1}\cdot \cdot \cdot a_{n})^{1/n}>0$.
Consequently, it follows from \cite{Sa} that if $\mu $ is the associated
(random) uniform Cantor measure, then almost surely $\widehat{\mu }\in \ell
^{p}$ for some $p<\infty $.

It only remains to check that for all such random Cantor sets $E,$ we have $%
\dimqL E=0$. This is also easy to verify. Just take $R$ to be the length of
the Cantor intervals at step $n_{j}$ in the construction, $r$ the length of
the Cantor intervals at step $n_{j}+m_{j}$ and $x$ to be an endpoint of a
step $n_{j}$ interval. Then there is a $\delta >0$ such that $r\leq
R^{1+\delta }$. As well, $N_{r}(B(x,R))=2^{m_{j}}$, while $R/r\geq (\log
n_{j})^{m_{j}}$.
\end{example}

\begin{example}
\label{LpI} An $L^{p}$-improving measure $\mu $ with $\dimqL\mu =0$ and $%
\dimL\supp\mu >0$: By modifying Salem's construction in \cite[Theorem 2]{Sa}
we can also give an example of an $L^{p}$-improving measure of quasi-lower
Assouad dimension zero, whose support has positive lower dimension.

We will put $a_{k}=1/4$, $b_{k}=1/4+1/\log k$ and construct the random
Cantor sets with ratio of dissection $\xi _{k}$ at step $k,$ as before.
Certainly all such sets will have positive lower Assouad dimension. The
random measure $\mu _{\omega }$ will be the weak$^{\ast }$ limit of the
measures 
\begin{equation*}
\mu _{\omega }^{(N)}=\prod\limits_{k=1}^{N}\left( p_{k}\delta
_{0}+(1-p_{k})\delta _{\xi _{1}\cdot \cdot \cdot \xi _{k-1}(1-\xi
_{k})}\right) ,
\end{equation*}%
where $p_{k}=1/j$ if $k=n_{j}+1,\dots,2n_{j}$ and $p_{k}=1/2$ otherwise.
Again, $\{n_{j}\}$ will be a very rapidly growing sequence with $n_{j+1}\gg
2n_{j}$. Note that 
\begin{equation*}
\left\vert \widehat{\mu _{\omega }^{(N)}}(n)\right\vert \leq \prod\limits 
_{\substack{ k=1  \\ k\notin \{n_{j}+1,\dots,2n_{j}\}}}^{N}\left\vert \cos
(\pi n\xi _{1}\cdot \cdot \cdot \xi _{k-1}(1-\xi _{k})\right\vert .
\end{equation*}%
Let $\varepsilon _{s}=\frac{1}{\pi }\int_{0}^{\pi }\left\vert \cos
x\right\vert ^{s}dx$ and temporarily fix integer $n$. Take $N=N(n)=\lfloor
\log \left\vert n\right\vert /\log 3\rfloor $ so 
\begin{equation*}
\left\vert n\right\vert a_{1}\cdot \cdot \cdot a_{N-1}/\log N=\left\vert
n\right\vert 3^{N-1}\geq n^{2}.
\end{equation*}%
By the same reasoning as in Salem's argument, 
\begin{equation*}
\int_{0}^{1}\left\vert \widehat{\mu _{\omega }^{(N)}}(n)\right\vert
^{s}d\omega \leq \prod_{\substack{ k-1  \\ k\notin \{n_{j}+1,\dots,2n_{j}\}}}%
^{N}\left( 1+\frac{1}{k^{2}}\right) \varepsilon _{s}^{N-M_{N}}
\end{equation*}%
where $M_{N}$ is the number of indices from the sets $\{n_{j}+1,\dots,2n_{j}%
\}$ that are at most $N$. If $\{n_{j}\}$ is sufficiently sparse and $N\in
(n_{J},n_{J+1}]$, then one can easily check that $N-M_{N}\geq N/6,$ thus%
\begin{equation*}
\int_{0}^{1}\left\vert \widehat{\mu _{\omega }^{(N)}}(n)\right\vert
^{s}d\omega \leq C\varepsilon _{s}^{\log \left\vert n\right\vert /(6\log 3)}
\end{equation*}%
for a universal constant $C$. Since $\varepsilon _{s}\rightarrow 0,$ we can
choose it so small that $\varepsilon _{s}^{\log \left\vert n\right\vert
/(6\log 3)}\leq n^{-2}$. For this choice of $s$, 
\begin{equation*}
\sum_{n=-\infty }^{\infty }\int_{0}^{1}\left\vert \widehat{\mu _{\omega }}%
(n)\right\vert ^{s}d\omega \leq \sum_{n=-\infty }^{\infty
}\int_{0}^{1}\left\vert \widehat{\mu _{\omega }^{(N)}}(n)\right\vert
^{s}d\omega \leq C\sum_{n=-\infty }^{\infty }n^{-2}<\infty \text{.}
\end{equation*}%
Consequently, the series $\sum_{n=-\infty }^{\infty }\int_{0}^{1}\left\vert 
\widehat{\mu _{\omega }}(n)\right\vert ^{s}d\omega $ converges and hence $%
\widehat{\mu _{\omega }}\in \ell ^{s}$ for a.e.\ $\omega $. Any such $\mu
_{\omega }$ is $L^{p}$-improving.

To see that $\dimqL\mu =0,$ consider $R$ the length of a Cantor interval of
step \thinspace $n_{j}+1$, $x$ its right hand endpoint and $r$ the length of
a Cantor interval of step $2n_{j}$. Then $R/r\geq 4^{n_{j}}$ while 
\begin{equation*}
\frac{\mu (B(x,R))}{\mu (B(x,r))}=\left( 1-\frac{1}{j}\right) ^{-n_{j}},
\end{equation*}%
from which it follows that $\dimqL\mu =0$. Being $L^{p}$-improving, $%
\inf_{x}\{\underline{\dim }_{\loc}\mu (x)\}>0$ and hence is not equal to $%
\dimqL\mu .$
\end{example}

\begin{remark}
The fact that there are measures $\mu $ with zero quasi-lower Assouad
dimension, but $\widehat{\mu }\in \ell ^{p}$ for some $p<\infty $ is
surprising in light of the general principle that one cannot have a measure
small in both its time and frequency domains.
\end{remark}


\section{The Assouad spectrum and quasi-Assouad dimensions of measures \label%
{sp}}

In \cite{FY1} and \cite{FY2}, Fraser and Yu introduced the notion of the
Assouad spectrum of a bounded set $E\subseteq \mathbb{R}^{d}$. These are the
functions 
\begin{equation*}
\theta \mapsto \dimAscottheta E=\inf \left\{ s:(\exists c)(\forall 0<R\leq 1)%
\text{ }\sup_{x\in E}N_{R^{1/\theta }}(B(x,R)\cap E)\leq c\left(
R^{1-1/\theta }\right) ^{s}\right\}
\end{equation*}%
and 
\begin{equation*}
\theta \mapsto\dimLscottheta E=\sup \left\{ s:(\exists c)(\forall 0<R\leq 1)%
\text{ }\sup_{x\in E}N_{R^{1/\theta }}(B(x,R)\cap E)\geq c\left(
R^{1-1/\theta }\right) ^{s}\right\}
\end{equation*}
for $\theta \in (0,1)$, which differ from the previously considered Assouad
dimensions by fixing the relationship of $r$ and $R$. In this section, we
study the corresponding notion for measures on fairly general metric spaces.

\begin{definition}
The \textbf{upper} and \textbf{lower Assouad spectrum} of the measure $\mu $
are the functions defined on $(0,1)$ by 
\begin{equation*}
\theta \mapsto\dimAscottheta\mu =\inf \left\{ s:(\exists c)\text{ }(\forall
0<R\leq 1)\sup_{x\in \supp\mu }\text{ }\frac{\mu (B(x,R))}{\mu
(B(x,R^{1/\theta }))}\leq c\left( R^{1-1/\theta }\right) ^{s}\right\} .
\end{equation*}
and 
\begin{equation*}
\theta \mapsto \dimLscottheta\mu =\sup \left\{ s:(\exists c)\text{ }(\forall
0<R\leq 1)\text{ }\inf_{x\in \supp\mu }\text{ }\frac{\mu (B(x,R))}{\mu
(B(x,R^{1/\theta }))}\geq c\left( R^{1-1/\theta }\right) ^{s}\right\} .
\end{equation*}
\end{definition}

This fixes the relationship of $r$ and $R$ as $r=R^{1/\theta}$. Another way
to define the spectrum is by only requiring an upper bound, i.e.\ we set $%
r\leq R^{1/\theta}$. These \textquotedblleft less than or
equal\textquotedblright\ spectra will be denoted by $\dimAcantheta$ and $%
\dimLcantheta$. Note that we have already defined this notion when
introducing the quasi-Assouad dimension and $\dimAcantheta \mu=\overline{H}%
(1/\theta-1)$ and $\dimLcantheta \mu=\underline{H}(1/\theta-1)$. Clearly,
for $\psi \leq \theta $ we have 
\begin{align*}
\dimqA\mu & \geq \dimAcantheta\mu \geq \dimAscotpsi\mu , \\
\dimqL\mu & \leq \dimLcantheta\mu\leq \dimLscotpsi\mu .
\end{align*}

In \cite{FHHTY} it was shown that $\overline{h}(1/\theta -1)=\sup_{0<\psi
\leq \theta }\dimAscotpsi E$ for subsets of $\mathbb{R}^{d},$ although the
same proof holds for any doubling metric space $E,$ i.e. spaces where $\dimA %
E<\infty $. Consequently, 
\begin{equation*}
\lim_{\theta \rightarrow 1}\dimAscottheta E =\limsup_{\theta \rightarrow 1}%
\dimAscottheta E=\dimqA E.
\end{equation*}%
The corresponding result was later proved for the quasi-lower Assouad
dimension in \cite{CWC} (with the additional assumption that the space $E$
was uniformly perfect). It is straightforward to obtain the analogous result
for doubling measures, that is measures $\mu$ for which $\dimA\mu <\infty $.
But this is a stringent condition for measures. However, it is possible to
obtain the same conclusion for measures which only satisfy the weaker
(quasi-doubling) condition, $\dimqA\mu <\infty $ and this we do in Theorem %
\ref{thm:spectrum} below. The general scheme of the proof is essentially the
same as in \cite{FHHTY}, but new technical complications arise. Examples of
such measures include equicontractive, self-similar measures that are
regular, meaning the probabilities associated with the right and left-most
similarities are equal and minimal. These measures are typically not
doubling if they fail the open set condition. For a proof that such measures
are quasi-doubling and specific examples of quasi-doubling, but not
doubling, measures, we refer the reader to \cite{HHT}.

\begin{theorem}
\label{thm:spectrum} Suppose $\mu $ is a probability measure and $\dimqA\mu
<\infty $. Let $\theta \in (0,1)$.

(i) Then 
\begin{equation*}
\dimAcantheta\mu =\sup_{0<\psi \leq \theta }\dimAscotpsi\mu \quad \text{ and 
}\quad \dimLcantheta\mu =\inf_{0<\psi \leq \theta }\dimLscotpsi\mu .
\end{equation*}

(ii) Moreover, $\lim_{\theta \rightarrow 1}\dimLscottheta\mu =\dimqL\mu $
and $\lim_{\theta \rightarrow 1}\dimAscottheta\mu =\dimqA\mu $.
\end{theorem}

We remark that, in particular, the quasi-Assouad dimensions of a doubling
measure can be recovered from the limiting behaviour of the Assouad spectrum.

The proof will proceed as follows: We first prove an elementary technical
result, followed by the proof of part (i) of the theorem. We will then show
that for quasi-doubling measures, the maps $\theta \mapsto\dimLscottheta$ or 
$\dimAscottheta$ are continuous for all $\theta \in (0,1)$. Lastly, this
fact will be used in proving part (ii) of the theorem.

\begin{lemma}
\label{thm:numTheo} Let $0<\beta <\theta <1$ and assume $\log \theta /\log
\beta \notin \mathbb{Q}$. Let 
\begin{equation*}
L=\{m\log \beta +n\log \theta :m,n\in \mathbb{N\}=\{}y_{j}\}_{j=1}^{\infty }
\end{equation*}%
where $y_{j}$ is ordered decreasingly (to $-\infty $). Then $%
\lim_{j\rightarrow \infty }y_{j}-y_{j+1}=0$. Furthermore, assume $(\theta
_{i})$ is a sequence tending to $0$ with $\theta _{i}>0$. Then, given small $%
\eta >0,$ there is an index $i_{0}$ such that for all $i\geq i_{0}$ there
exist positive integers $m,n$ such that 
\begin{equation}
\frac{\eta }{2\theta _{i}}\leq \frac{1}{\theta _{i}}-\frac{1}{\theta
^{n}\beta ^{m}}\leq \frac{4\eta }{\theta _{i}}.  \label{*}
\end{equation}
\end{lemma}

\begin{proof}
The fact that $y_{j}-y_{j+1}\rightarrow 0$ is well known, so we will only
prove the second statement.

Fix small $0<\eta <1$ and choose $J$ such that $|y_{j}-y_{j+1}|<\eta $ for
all $j\geq J$. Choose $i_{0}$ such that $\log 1/\theta _{i}\geq |y_{J}|+1$
for all $i\geq i_{0}$. Temporarily fix such an $i$ and choose the maximal
index $k$ such that $\log 1/\theta _{i}>|y_{k}|+\eta $ . Note that $k\geq J$%
. As $k$ is maximal, it must be that either $|y_{k+1}|\geq \log 1/\theta
_{i} $ or $0<\log 1/\theta _{i}-|y_{k+1}|<\eta $. In either case, the fact
that $|y_{k}-y_{k+1}|<\eta $ ensures that 
\begin{equation*}
\eta <\log 1/\theta _{i}-|y_{k}|<2\eta .
\end{equation*}%
It is now a routine calculation to see that if $y_{k}=m\log \beta +n\log
\theta $, then $e^{\eta }<\beta ^{m}\theta ^{n}/\theta _{i}<e^{2\eta }$.
Thus for all $i\geq i_{0},$ we have%
\begin{equation*}
\frac{\eta }{2\theta _{i}}\leq \frac{\eta e^{-2\eta }}{\theta _{i}}\leq 
\frac{\eta }{\beta ^{m}\theta ^{n}}<\frac{1}{\theta _{i}}-\frac{1}{\beta
^{m}\theta ^{n}}<\frac{4\eta }{\beta ^{m}\theta ^{n}}\leq \frac{4\eta }{%
\theta _{i}}.\qedhere
\end{equation*}
\end{proof}

\begin{proof}[Proof of Theorem~\protect\ref{thm:spectrum}]
(i).$\;\;$ First, consider the upper Assouad spectrum. There is no loss in
assuming $s=\dimAcantheta\mu >0$ for otherwise $\dimAscotpsi\mu =0$ for all $%
\psi \leq \theta $ as well. Fix $0<\varepsilon <s$ and obtain $x_{i}\in \supp%
\mu ,R_{i}\rightarrow 0$ and $r_{i}=R_{i}^{1/\theta _{i}}\leq R_{i}$, with $%
\theta _{i}\leq \theta $ and%
\begin{equation*}
\frac{\mu (B(x_{i},R_{i}))}{\mu (B(x_{i},R_{i}^{1/\theta _{i}}))}\geq \left( 
\frac{R_{i}}{R_{i}^{1/\theta _{i}}}\right) ^{s-\varepsilon }.
\end{equation*}%
Without loss of generality, we can assume $\theta _{i}\rightarrow \psi $
where $\psi \in \lbrack 0,\theta ]$ and that the convergence is monotonic.

Case 1: We will first assume $\psi >0$. If $(\theta _{i})$ is a decreasing
sequence (so $R_{i}^{1/\theta _{i}}\geq R_{i}^{1/\psi }$), then we have%
\begin{equation*}
\frac{\mu (B(x_{i},R_{i}))}{\mu (B(x_{i},R_{i}^{1/\psi }))}\geq \frac{\mu
(B(x_{i},R_{i}))}{\mu (B(x_{i},R_{i}^{1/\theta _{i}}))}\geq \left( \frac{%
R_{i}}{R_{i}^{1/\theta _{i}}}\right) ^{s-\varepsilon }.
\end{equation*}%
As $1-1/\theta _{i}\rightarrow 1-1/\psi ,$ it follows that 
\begin{equation*}
\frac{\mu (B(x_{i},R_{i}))}{\mu (B(x_{i},R_{i}^{1/\psi }))}\geq
R_{i}^{(1-1/\psi )\left( \frac{1-\theta _{i}}{1-\psi }\right) (s-\varepsilon
)}\geq R_{i}^{(1-1/\psi )(s-\varepsilon /2)}
\end{equation*}%
if $i$ is sufficiently large. That implies $\dimAscotpsi\mu \geq
s-\varepsilon /2$ and as $\varepsilon >0$ was arbitrary we deduce that $%
\dimAscotpsi\mu \geq s$. Thus $\sup_{0<\psi \leq \theta }\dimAscotpsi\mu =s.$

Otherwise, we can assume $(\theta _{i})$ increases to $\psi \leq \theta <1$.
Choose $n_{i}\in \mathbb{N}$ so that 
\begin{equation*}
2^{-(n_{i}+1)}<R_{i}^{1/\psi }\leq 2^{-n_{i}}
\end{equation*}%
(where we choose a subsequence of $\{R_{i}\},$ if necessary, to ensure the
sequence $\{n_{i}\}$ is strictly increasing) and define a function $g$ on $%
\mathbb{N}$ by $g(n)=R_{i}^{1/\theta _{i}}2^{n_{i}}$ if $n_{i}\leq n<n_{i+1}$%
. Then $\log R_{i}\sim n_{i}$ and%
\begin{equation*}
\log 1/2+\left( \frac{1}{\theta _{i}}-\frac{1}{\psi }\right) \log R_{i}\leq
\log g(n_{i})\leq \left( \frac{1}{\theta _{i}}-\frac{1}{\psi }\right) \log
R_{i}.
\end{equation*}
Hence if $n_{i}\leq n<n_{i+1},$ 
\begin{equation*}
\frac{\left\vert \log g(n)\right\vert }{n}\leq \frac{\left\vert \log
g(n_{i})\right\vert }{n_{i}}\rightarrow 0\text{ as }n\rightarrow \infty 
\text{ (equivalently, }i\rightarrow \infty ).
\end{equation*}

As proven in \cite[Proposition 4.2]{HHT}, the assumption that $\dimqA\mu
<\infty $ ensures that for each $q>1$ there is a constant $c$ such that for
all $i$, 
\begin{equation*}
\mu (B(x_{i},R_{i}^{1/\theta _{i}}))=\mu (B(x_{i},g(n_{i})2^{-n_{i}}))\geq
cq^{-n_{i}}\mu (B(x_{i},2^{-n_{i}}))\geq cq^{-n_{i}}\mu
(B(x_{i},R_{i}^{1/\psi })).
\end{equation*}%
Thus 
\begin{eqnarray*}
\left( \frac{R_{i}}{R_{i}^{1/\theta _{i}}}\right) ^{s-\varepsilon } &\leq &%
\frac{\mu (B(x_{i},R_{i}))}{\mu (B(x_{i},R_{i}^{1/\theta _{i}}))}=\frac{\mu
(B(x_{i},R_{i}))}{\mu (B(x_{i},R_{i}^{1/\psi }))}\frac{\mu
(B(x_{i},R_{i}^{1/\psi }))}{\mu (B(x_{i},R_{i}^{1/\theta _{i}}))} \\
&\leq &\frac{q^{n_{i}}}{c}\frac{\mu (B(x_{i},R_{i}))}{\mu
(B(x_{i},R_{i}^{1/\psi }))}.
\end{eqnarray*}%
As 
\begin{equation*}
\left( R_{i}^{1/\psi }\right) ^{\log q/\log 2}\leq \left( 2^{-n_{i}}\right)
^{\log q/\log 2}=q^{-n_{i}},
\end{equation*}%
that shows%
\begin{equation*}
\frac{\mu (B(x_{i},R_{i}))}{\mu (B(x_{i},R_{i}^{1/\psi }))}\geq
cR_{i}^{(1-1/\theta _{i})(s-\varepsilon )}q^{-n_{i}}\geq cR_{i}^{(1-1/\theta
_{i})(s-\varepsilon )}R_{i}^{\log q/(\psi \log 2)}=cR_{i}^{(1-1/\psi )t_{i}}
\end{equation*}%
where 
\begin{equation*}
t_{i}=\left( \frac{1-1/\theta _{i}}{1-1/\psi }\right) (s-\varepsilon )+\frac{%
\log q}{(\psi -1)\log 2}.
\end{equation*}%
Since $\theta _{i}\rightarrow \psi \in (0,1)$ as $i\rightarrow \infty ,$ 
\begin{equation*}
t_{i}\rightarrow s-\varepsilon -\frac{\log q}{(1-\psi )\log 2}.
\end{equation*}%
As $q>1$ and $\varepsilon >0$ are arbitrary, we again deduce that $%
\dimAscotpsi\geq s$ and that gives the desired result.

\vskip.5em Case 2: Now suppose $\psi =0$. We will make use of Claim~\ref%
{thm:numTheo} and choose $\beta \in (0,\theta )$ such that $\log \theta
/\log \beta $ is irrational. Suppose for a contraction that 
\begin{equation*}
\max \{\dimAscottheta,\dimAscotbeta\}\leq s-3\varepsilon \text{.}
\end{equation*}%
For all small enough $R$ and $x\in \supp\mu $ we have for $\gamma =\theta
,\beta ,$%
\begin{equation}
\frac{\mu (B(x_{i},R_{i}))}{\mu (B(x_{i},R_{i}^{1/\gamma }))}\leq \left( 
\frac{R_{i}}{R_{i}^{1/\gamma }}\right) ^{s-2\varepsilon }.  \label{basic}
\end{equation}%
Fix $\eta >0$ small, to be specified later. Choose $m,n\in \mathbb{N}$ as in %
\eqref{*} with this choice of $\eta $. By repeated application of (\ref%
{basic}) and a telescoping argument, we see that%
\begin{equation*}
\frac{\mu (B(x_{i},R_{i}))}{\mu (B(x_{i},R_{i}^{1/\theta ^{n}\beta ^{m}}))}%
\leq \left( \frac{R_{i}}{R_{i}^{1/\theta ^{n}\beta ^{m}}}\right)
^{s-2\varepsilon }.
\end{equation*}%
The second part of the claim yields that%
\begin{equation*}
\frac{1}{\theta _{i}}\left( 1-\eta /2\right) \geq \frac{1}{\theta ^{n}\beta
^{m}}
\end{equation*}%
for all $i$ sufficiently large. Thus $R_{i}^{1/\theta _{i}}\leq \left(
R_{i}^{1/(\theta ^{n}\beta ^{m})}\right) ^{1/(1-\eta /2)}$. It follows that
if $d=\dimqA\mu $, $\varepsilon >0$ and $\eta $ is sufficiently small, there
is a constant $c$ such that 
\begin{equation*}
\frac{\mu (B(x_{i},R_{i}^{1/\theta ^{n}\beta ^{m}}))}{\mu
(B(x_{i},R_{i}^{1/\theta _{i}}))}\;\leq \;c\left( \frac{R_{i}^{1/\theta
^{n}\beta ^{m}}}{R_{i}^{1/\theta _{i}}}\right) ^{d+\varepsilon }.
\end{equation*}%
Thus%
\begin{eqnarray*}
\left( \frac{R_{i}}{R_{i}^{1/\theta _{i}}}\right) ^{s-\varepsilon } &\leq &%
\frac{\mu (B(x_{i},R_{i}))}{\mu (B(x_{i},R_{i}^{1/\theta _{i}}))}\quad \leq
\quad \frac{\mu (B(x_{i},R_{i}))}{\mu (B(x_{i},R_{i}^{1/\theta ^{n}\beta
^{m}}))}\frac{\mu (B(x_{i},R_{i}^{1/\theta ^{n}\beta ^{m}}))}{\mu
(B(x_{i},R_{i}^{1/\theta _{i}}))} \\
&\leq &c\left( \frac{R_{i}}{R_{i}^{1/\theta ^{n}\beta ^{m}}}\right)
^{s-2\varepsilon }\left( \frac{R_{i}^{1/\theta ^{n}\beta ^{m}}}{%
R_{i}^{1/\theta _{i}}}\right) ^{d+\varepsilon },
\end{eqnarray*}%
and this implies that if we put 
\begin{eqnarray*}
t_{i} &=&(\frac{1}{\theta _{i}}-1)(s-\varepsilon )+(1-\frac{1}{\theta
^{n}\beta ^{m}})(s-2\varepsilon )+(\frac{1}{\theta ^{n}\beta ^{m}}-\frac{1}{%
\theta _{i}})(d+\varepsilon ) \\
&=&-\varepsilon +(s-\varepsilon )(\frac{1}{\theta _{i}}-\frac{1}{\theta
^{n}\beta ^{m}})+\frac{\varepsilon }{\theta ^{n}\beta ^{m}}+(\frac{1}{\theta
^{n}\beta ^{m}}-\frac{1}{\theta _{i}})(d+\varepsilon ),
\end{eqnarray*}%
then 
\begin{equation}
cR_{i}^{t_{i}}\geq 1\;\text{ for all large }i.  \label{bd}
\end{equation}%
Using the bounds from \eqref{*} we deduce that for small enough $\eta ,$ 
\begin{eqnarray*}
t_{i} &\geq &-\varepsilon +(s-\varepsilon )\eta /(2\theta _{i})+\varepsilon
(1-4\eta )/\theta _{i}-4\eta (d+\varepsilon )/\theta _{i} \\
&=&-\varepsilon +\frac{1}{\theta _{i}}((s-\varepsilon )\eta /2+\varepsilon
(1-4\eta )-4\eta (d+\varepsilon ))\geq -\varepsilon +\varepsilon /(2\theta
_{i})\rightarrow \infty
\end{eqnarray*}%
as $i\rightarrow \infty $. But that means $R_{i}^{t_{i}}\rightarrow 0$ and
hence \eqref{bd} cannot be satisfied for all large $i$ (whatever the choice
of constant $c$). This proves the result for the upper Assouad spectrum.

\vskip1em We now turn to the proof for the lower Assouad spectrum. Let $s=%
\dimLcantheta\mu \leq \dimqA\mu <\infty $. Fix $\varepsilon >0$ and obtain $%
x_{i}\in \supp\mu ,R_{i}\rightarrow 0$ and $r_{i}=R_{i}^{1/\theta _{i}}\leq
R_{i}$, with $\theta _{i}\leq \theta $ and%
\begin{equation*}
\frac{\mu (B(x_{i},R_{i}))}{\mu (B(x_{i},R_{i}^{1/\theta _{i}}))}\leq \left( 
\frac{R_{i}}{R_{i}^{1/\theta _{i}}}\right) ^{s+\varepsilon }.
\end{equation*}%
As before, without loss of generality we can assume $\theta _{i}\rightarrow
\psi $ where $\psi \in \lbrack 0,\theta ]$ and that the convergence is
monotonic.

\vskip.5em Case 1: We will first assume $\psi >0$. If $(\theta _{i})$ is an
increasing sequence (so $R_{i}^{1/\theta _{i}}\leq R_{i}^{1/\psi }$), then,
similar to the first step in the upper Assouad spectrum argument, we have%
\begin{equation*}
\frac{\mu (B(x_{i},R_{i}))}{\mu (B(x_{i},R_{i}^{1/\psi }))}\leq \frac{\mu
(B(x_{i},R_{i}))}{\mu (B(x_{i},R_{i}^{1/\theta _{i}}))}\leq \left( \frac{%
R_{i}}{R_{i}^{1/\theta _{i}}}\right) ^{s+\varepsilon }\leq \left( \frac{R_{i}%
}{R_{i}^{1/\psi }}\right) ^{s+2\varepsilon }
\end{equation*}%
for large enough $i$. That implies $\dimLscotpsi\mu \leq s+2\varepsilon $
and hence $\inf_{0<\psi \leq \theta }\dimLscotpsi\mu =s.$

Now suppose $(\theta _{i})$ decreases to $\psi $. As each $\theta _{i}\leq
\theta <1,$ the same is true for $\psi $ and $R_{1}^{1/\theta _{i}}\geq
R_{i}^{1/\psi }$. In a similar fashion to the second step in case 1 above,
we choose $n_{i}$ so that $2^{-(n_{i}+1)}<R_{i}^{1/\theta _{i}}\leq
2^{-n_{i}}$ and define $g$ by $g(n)=R_{i}^{1/\psi }2^{n_{i}}$ if $n_{i}\leq
n<n_{i+1}$. As before, one can easily check that $\log g(n)/n\rightarrow 0,$
hence the fact that $\dimqA\mu <\infty $ implies that for any fixed $q>1$
and suitable constant $c$ we have%
\begin{equation*}
\frac{\mu (B(x_{i},R_{i}))}{\mu (B(x_{i},R_{i}^{1/\psi }))}\leq
cq^{n_{i}}\left( \frac{R_{i}}{R_{i}^{1/\theta _{i}}}\right) ^{s+\varepsilon
}\leq cR_{i}^{-\log q/(\theta _{i}\log 2)}R_{i}^{(1-1/\theta
_{i})(s+\varepsilon )}=R_{i}^{(1-1/\psi )t_{i}}
\end{equation*}%
for 
\begin{equation*}
t_{i}=\frac{-\log q}{\theta _{i}\log 2(1-1/\psi )}+\frac{(1-1/\theta
_{i})(s+\varepsilon )}{1-1/\psi }\rightarrow \frac{\log q}{\log 2(1-\psi )}%
+s+\varepsilon .
\end{equation*}%
Since $q>1$ and $\varepsilon >0$ are arbitrary, we deduce that $\inf_{0<\psi
\leq \theta }\dimLscotpsi\mu =s.$

\vskip.5em Case 2: Now suppose $\psi =0$. Choose $0<\beta <\theta <1$ with $%
\log \theta /\log \beta \notin \mathbb{Q}$ and suppose for a contradiction
that 
\begin{equation*}
\min \{\dimLscottheta\mu ,\dimLscotbeta\mu \}\geq s+3\varepsilon \text{.}
\end{equation*}%
Then for all small enough $R$, $x\in \supp\mu $ and $\gamma =\theta ,\beta $
we have%
\begin{equation*}
\frac{\mu (B(x,R))}{\mu (B(x,R^{1/\gamma }))}\geq R^{(1-1/\gamma
)(s+2\varepsilon )}
\end{equation*}%
Fix $\eta >0$ small, to be specified later and choose $m,n\in \mathbb{N}$ as
in \eqref{*} with this choice of $\eta $. A telescoping argument gives%
\begin{equation*}
\frac{\mu (B(x_{i},R_{i}))}{\mu (B(x_{i},R_{i}^{1/\theta ^{n}\beta ^{m}}))}%
\geq \left( \frac{R_{i}}{R_{i}^{1/\theta ^{n}\beta ^{m}}}\right)
^{s+2\varepsilon }.
\end{equation*}%
Since $1/\theta _{i}\geq 1/(\theta ^{n}\beta ^{m}),R^{1/\theta _{i}}\leq
R_{i}^{1/\theta ^{n}\beta ^{m}}$ and therefore%
\begin{equation*}
\left( \frac{R_{i}}{R_{i}^{1/\theta _{i}}}\right) ^{s+\varepsilon }\;\geq \;%
\frac{\mu (B(x_{i},R_{i}))}{\mu (B(x_{i},R_{i}^{1/\theta _{i}}))}\;\geq \;%
\frac{\mu (B(x_{i},R_{i}))}{\mu (B(x_{i},R_{i}^{1/\theta ^{n}\beta ^{m}}))}%
\;\geq \;\left( \frac{R_{i}}{R_{i}^{1/\theta ^{n}\beta ^{m}}}\right)
^{s+2\varepsilon }.
\end{equation*}%
Equivalently, $1\geq cR_{i}^{t_{i}}$ for all large $i$ where 
\begin{equation*}
t_{i}=\left( \frac{1}{\theta _{i}}-1\right) (s+\varepsilon )+\left( 1-\frac{1%
}{\theta ^{n}\beta ^{m}}\right) (s+2\varepsilon ).
\end{equation*}%
But the properties of $\theta $ and $\beta $ ensure that for small enough $%
\eta ,t_{i}\leq \varepsilon -\varepsilon /(2\theta _{i})\rightarrow -\infty $
and that's a contradiction.

This completes the proof of the lower Assouad spectrum result and thus part
(i).
\end{proof}

\begin{lemma}
Assume $\dimqA\mu <\infty $. Then for each $\theta \in (0,1)$ the functions $%
\theta \mapsto\dimLscottheta\mu $ and $\theta\mapsto\dimAscottheta\mu $ are
continuous.
\end{lemma}

\begin{proof}
We will give the complete proof for the continuity of the lower Assouad
spectrum and leave the analogous proof of the continuity of the upper
Assouad spectrum to the reader.

Fix $\theta \in (0,1)$ and let $t=\dimLscottheta\mu $. We proceed by
contradiction. If $\theta \mapsto\dimLscottheta\mu $ is not continuous at $%
\theta ,$ then there is some $\varepsilon >0$ and a sequence $\theta
_{j}\rightarrow \theta $ such that $\vert \dimLscottheta\mu -\dimLscotthetaj%
\mu \vert \geq 3\varepsilon $ for all $j$.

Suppose that there is a subsequence such that $\theta _{j}\rightarrow \theta 
$ and $\dimLscotthetaj\mu \geq t+3\varepsilon $ for all $j$. Then for each $%
j $ there is some $R(j)>0$ such that for all $R\leq R(j)$ and for each $x\in %
\supp\mu $ we have%
\begin{equation}
\frac{\mu (B(x,R))}{\mu (B(x,R^{1/\theta _{j}}))}\geq R^{(1-1/\theta
_{j})(t+2\varepsilon )}.  \label{C1}
\end{equation}

If a further subsequence satisfies $\theta _{j}\geq \theta $ for all $j$,
then fix small $\delta >0$ and choose $j$ such that $\left\vert 1-1/\theta
_{j}\right\vert \geq (1-\delta )\left\vert 1-1/\theta \right\vert $. Since $%
R^{1/\theta _{j}}\geq R^{1/\theta },$ we have%
\begin{equation*}
\frac{\mu (B(x,R))}{\mu (B(x,R^{1/\theta }))}\geq \frac{\mu (B(x,R))}{\mu
(B(x,R^{1/\theta _{j}}))}\geq R^{(1-1/\theta _{j})(t+2\varepsilon )}\geq
R^{(1-1/\theta )(t+2\varepsilon )(1-\delta )}\geq R^{(1-1/\theta
)(t+\varepsilon )},
\end{equation*}%
for all $x$ and $R\leq R(j),$ provided we choose $\delta $ small enough, and
that contradicts the assumption that $t=\dimLscottheta\mu $.

So assume that $\theta _{j}\leq \theta $ for all $j$. Since $t=\dimLscottheta%
\mu ,$ we can choose a sequence $x_{j}$ and $R_{j}\leq R(j),$ $%
R_{j}\rightarrow 0$, such that 
\begin{equation}
\frac{\mu (B(x_{j},R_{j}))}{\mu (B(x_{j},R_{j}^{1/\theta }))}\leq
R_{j}^{(1-1/\theta )(t+\varepsilon /4)}.  \label{C3}
\end{equation}%
By passing to a further subsequence, if necessary, we can assume there is a
sequence of integers, $(n_{j}),$ strictly increasing to infinity, such that 
\begin{equation*}
2^{-(n_{j}+1)}<R_{j}^{1/\theta }\leq 2^{-n_{j}}.
\end{equation*}%
Define 
\begin{equation*}
g(n)=R_{j}^{1/\theta _{j}}2^{n_{j}}\text{ if }n\in \lbrack n_{j},n_{j+1}).
\end{equation*}%
As in the proof of the first part of theorem, $\log g(n)/n\rightarrow 0,$
thus by the quasi-doubling property of $\mu ,$ (the assumption that $\dimqA%
\mu <\infty )$ for each $q>1$ there is a constant $C_{q}$ such that%
\begin{eqnarray*}
\mu (B(x_{j},R_{j}^{1/\theta })) &\leq &\mu (B(x_{j},2^{-n_{j}}))\leq
C_{q}q^{n_{j}}\mu (B(x_{j},g(n_{j})2^{-n_{j}}))=C_{q}q^{n_{j}}\mu
(B(x_{j},R_{j}^{1/\theta _{j}})) \\
&\leq &C_{q}R_{j}^{-\frac{\log q}{\theta \log 2}}\mu
(B(x_{j},R_{j}^{1/\theta _{j}})).
\end{eqnarray*}%
Choose $q>1$ such that $\log q/(\theta \log 2)\leq (1/\theta -1)\varepsilon
/4$. With this fixed choice of $q$ we see that%
\begin{equation*}
\mu (B(x_{j},R_{j}^{1/\theta }))\leq C_{q}R_{j}^{(1-1/\theta )\varepsilon
/4}\mu (B(x_{j},R_{j}^{1/\theta _{j}})).
\end{equation*}%
Together with (\ref{C3}) we deduce that%
\begin{eqnarray}
R_{j}^{(1-1/\theta )(t+\varepsilon /4)} &\geq &\frac{\mu (B(x_{j},R_{j}))}{%
\mu (B(x_{j},R_{j}^{1/\theta }))}=\frac{\mu (B(x_{j},R_{j}))}{\mu
(B(x_{j},R_{j}^{1/\theta _{j}}))}\frac{\mu (B(x_{j},R_{j}^{1/\theta _{j}}))}{%
\mu (B(x_{j},R_{j}^{1/\theta }))}  \label{C2} \\
&\geq &\frac{\mu (B(x_{j},R_{j}))}{\mu (B(x_{j},R_{j}^{1/\theta _{j}}))}%
C_{q}^{-1}R_{j}^{-(1-1/\theta )\varepsilon /4}.  \notag
\end{eqnarray}%
Now choose $j_{1}$ such that for all $j\geq j_{1},$ 
\begin{equation*}
\frac{1-1/\theta _{j}}{1-1/\theta }(t+2\varepsilon )\geq t+\varepsilon .
\end{equation*}%
Combining (\ref{C1}) and (\ref{C2}) gives%
\begin{equation*}
R_{j}^{(1-1/\theta )(t+\varepsilon )}\leq R_{j}^{(1-1/\theta
_{j})(t+2\varepsilon )}\leq \frac{\mu (B(x_{j},R_{j}))}{\mu
(B(x_{j},R_{j}^{1/\theta _{j}}))}\leq C_{q}R_{j}^{(1-1/\theta
)(t+\varepsilon /2)}
\end{equation*}%
for all $j\geq j_{1}$. But since $C_{q}$ is fixed, the outer most
inequalities clearly cannot hold for all $R_{j}\rightarrow 0.$ This
contradiction shows that we cannot have a sequence $\theta _{j}\rightarrow
\theta $ such that $\dimLscotthetaj\mu \geq t+3\varepsilon $ for all $j$.

Otherwise, there must be a subsequence such that $\theta _{j}\rightarrow
\theta $ and $\dimLscotthetaj\mu \leq t-3\varepsilon $ for all $j$. In this
case, there must be $x_{j}\in \supp
\mu $ and a decreasing sequence $R_{j}\rightarrow 0$ such that%
\begin{equation}
\frac{\mu (B(x_{j},R_{j}))}{\mu (B(x_{j},R_{j}^{1/\theta _{j}}))}\leq
R_{j}^{(1-1/\theta _{j})(t-2\varepsilon )}  \label{C6}
\end{equation}%
$\ $for all $j$.

If $\theta _{j}\leq \theta ,$ then as $R_{j}^{1/\theta _{j}}\leq
R_{j}^{1/\theta },$ 
\begin{equation*}
\frac{\mu (B(x_{j},R_{j}))}{\mu (B(x_{j},R_{j}^{1/\theta }))}\leq \frac{\mu
(B(x_{j},R_{j}))}{\mu (B(x_{j},R_{j}^{1/\theta _{j}}))}\leq
R_{j}^{(1-1/\theta _{j})(t-2\varepsilon )}\leq R_{j}^{(1-1/\theta
)(t-\varepsilon )}
\end{equation*}%
for $j$ sufficiently large ($R_{j}$ small) and that contradicts the
assumption that $t=\dimLscottheta\mu $.

So we may assume $\theta _{j}\geq \theta $. The arguments are similar to the
case $\theta _{j}\leq \theta $ above. Without loss of generality, there is a
strictly increasing sequence $(n_{j})$ satisfying 
\begin{equation*}
2^{-(n_{j}+1)}<R_{j}^{1/\theta _{j}}\leq 2^{-n_{j}}.
\end{equation*}%
Put 
\begin{equation*}
g(n)=R_{j}^{1/\theta }2^{n_{j}}\text{ if }n\in \lbrack n_{j},n_{j+1}).
\end{equation*}%
The quasi-doubling property of $\mu $ ensures that for each $q>1$ there is a
constant $c_{q}>0$ such that%
\begin{eqnarray*}
\mu (B(x_{j},R_{j}^{1/\theta })) &=&\mu (B(x_{j},g(n_{j})2^{-n_{j}}))\geq
c_{q}q^{-n_{j}}\mu (B(x_{j},2^{-n_{j}})) \\
&\geq &c_{q}q^{-n_{j}}\mu (B(x_{j},R_{j}^{1/\theta _{j}}))\geq c_{q}R_{j}^{%
\frac{\log q}{\theta _{j}\log 2}}\mu (B(x_{j},R_{j}^{1/\theta _{j}})).
\end{eqnarray*}%
Hence from (\ref{C6}),%
\begin{equation*}
R_{j}^{(1-1/\theta _{j})(t-2\varepsilon )}\geq \frac{\mu (B(x_{j},R_{j}))}{%
\mu (B(x_{j},R_{j}^{1/\theta _{j}}))}=\frac{\mu (B(x_{j},R_{j}))}{\mu
(B(x_{j},R_{j}^{1/\theta }))}\frac{\mu (B(x_{j},R_{j}^{1/\theta }))}{\mu
(B(x_{j},R_{j}^{1/\theta _{j}}))}\geq c_{q}R_{j}^{\frac{\log q}{\theta
_{j}\log 2}}\frac{\mu (B(x_{j},R_{j}))}{\mu (B(x_{j},R_{j}^{1/\theta }))},
\end{equation*}%
so that%
\begin{equation*}
\frac{\mu (B(x_{j},R_{j}))}{\mu (B(x_{j},R_{j}^{1/\theta }))}\leq
c_{q}^{-1}R_{j}^{(1-1/\theta _{j})(t-2\varepsilon )}R_{j}^{\frac{-\log q}{%
\theta _{j}\log 2}}=c_{q}^{-1}R_{j}^{(1-1/\theta )\left( (t-2\varepsilon
)\left( \frac{1-1/\theta _{j}}{1-1/\theta }\right) -\frac{\log q}{\theta
_{j}(1-1/\theta )\log 2}\right) }.
\end{equation*}%
Choose $j_{2}$ such that for all $j\geq j_{2}$ we have $(t-2\varepsilon
)\left( 1-1/\theta _{j}\right) /(1-1/\theta )\leq t-\varepsilon $ and $%
\left\vert \theta _{j}(1-1/\theta )\right\vert \geq \left\vert \theta
-1\right\vert /2$ and then choose $q$ sufficiently close to $1$ so that $%
2\log q/(\left\vert \theta -1\right\vert \log 2)\leq \varepsilon /2$. We
conclude that for $j\geq j_{2},$%
\begin{equation*}
\frac{\mu (B(x_{j},R_{j}))}{\mu (B(x_{j},R_{j}^{1/\theta }))}\leq
c_{q}^{-1}R_{j}^{(1-1/\theta )(t-\varepsilon /2)}
\end{equation*}%
and, again, this contradicts the assumption that $t=\dimLscottheta\mu $.

That completes the proof for the continuity of the lower Assouad spectrum.
\end{proof}

We are now ready to complete the proof of the theorem.

\begin{proof}[Proof of Theorem \protect\ref{thm:spectrum}]
(ii).$\;\;$ It follows directly from the first part of the theorem that 
\begin{equation*}
\liminf_{\theta \rightarrow 1}\dimLscottheta\mu =\dimqL\mu \text{ and }%
\limsup_{\theta \rightarrow 1}\dimAscottheta\mu =\dimqA\mu .
\end{equation*}
Furthermore, an immediate consequence of the factorization%
\begin{equation*}
\frac{\mu (B(x,R))}{\mu (B(x,R^{1/\theta }))}=\frac{\mu (B(x,R))}{\mu
(B(x,R^{1/\theta ^{1/n}}))}\frac{\mu (B(x,R^{1/\theta ^{1/n}}))}{\mu
(B(x,R^{1/\theta ^{2/n}}))}\cdot \cdot \cdot \frac{\mu (B(x,R^{1/\theta
^{(n-1)/n}}))}{\mu (B(x,R^{1/\theta }))}
\end{equation*}%
is that $\underline{\dim}_{\,A}^{\,=\theta^{1/n}}\mu \leq \dimLscottheta\mu $
and $\overline{\dim} _{A}^{=\theta ^{1/n}}\mu \geq \dimAscottheta\mu $ for
all $n\in \mathbb{N}$.

These observations, together with the continuity result of the previous
lemma, allow one to use the same argument as given directly after Lemma 3.1
in \cite{FHHTY} to show that $\liminf_{\theta \rightarrow 1}\dimLscottheta%
\mu =\lim_{\theta \rightarrow 1}\dimLscottheta\mu $ and similarly for the
upper Assouad spectrum.

That completes the proof of Theorem \ref{thm:spectrum}.
\end{proof}

We can use Theorem~\ref{thm:spectrum} to state an analogue of the notion of
uniformly perfect for the quasi-lower Assouad dimension.

\begin{corollary}
Suppose $\mu $ is doubling. If there exists $t>0$ such that for every $%
\delta <0,$ all $x\in \supp\mu $ and all sufficiently small $R,$

\begin{equation*}
\mu (B(x,R)\setminus B(x,R^{1+\delta }))\geq (1-R^{\delta t})\mu (B(x,R)),
\end{equation*}%
then $\dimqL\mu \geq t$.
\end{corollary}

\begin{proof}
The hypothesis of the corollary is equivalent to the statement 
\begin{equation*}
\frac{\mu (B(x,R))}{\mu (B(x,R^{1+\delta }))}\geq R^{-\delta t}.
\end{equation*}%
Thus $\overline{H}(1/\delta-1) \geq t$ for all $\delta >0$ and hence $%
\dimLcantheta\mu\geq t$ for all $\theta>0$. As $\mu $ is doubling, $\dimqL%
\mu =\inf_{\theta >0}\dimLcantheta \mu \geq t$.
\end{proof}

\subsection*{Acknowledgements}

Both authors thank F. Mendivil and Acadia University for hosting a
sabbatical visit for K.H.\ and a short-term visit for S.T.\ when some of
this material was developed. We also thank K.G.~Hare for showing us Claim~%
\ref{thm:numTheo}.

\clearpage
\appendix

\section{Lower Spectrum for Sets}

In \cite{CWC} it was shown that $\dimLcantheta:=\underline{h}(1/\theta
-1)=\inf_{0\leq\psi\leq\theta}\dimLscotpsi$ under the assumption that the
metric space is doubling and uniformly perfect. In this section we will give
a shorter proof of this fact that does not require the uniformly perfect
assumption.

To begin, we note that if a metric space $X$ is doubling, then there is a
constant $c$ such that for all $x,r,R$ and subsets $E\subseteq X$, 
\begin{equation*}
N_{r}(B(x,R)\cap E)\leq N_{16r}(B(x,R)\cap E)\sup_{y}N_{r}(B(y,16r))\leq
cN_{16r}(B(x,r)\cap E).
\end{equation*}%
For any subset $F,$ let $M_{r}(F)$ be the maximal number of disjoint balls
of radius $r,$ centred in $F$. Since we have 
\begin{equation*}
N_{16r}(F)\leq M_{4r}(F)\leq N_{4r}(F)\leq M_{r}(F)\leq N_{r}(F)\leq
cN_{16r}(F),
\end{equation*}%
we can replace the covering numbers in the definition of the Assouad
spectrum and dimensions with packing numbers.

We will also require the following observation:

\begin{lemma}
\label{packing}For $0<r=r_{k}<r_{k-1}<\cdot \cdot \cdot <r_{1}<R$, 
\begin{equation*}
M_{r}(B(x,R))\geq
M_{r_{1}}(B(x,R-r_{1}))\inf_{y_{1}}M_{r_{2}}(B(y_{1},r_{1}-r_{2}))\cdot
\cdot \cdot \inf_{y_{k}}M_{r}(B(y_{1},r_{k-1}-r)).
\end{equation*}

\begin{proof}
Let $t_{1}=\inf_{y_{1}}M_{r_{2}}(B(y_{1},r_{1}-r_{2}))$ and suppose $%
\{B(x_{j},r_{1}):$ $j=1,\dots,J\}$ is a set of disjoint balls with centres
in $B(x,R-r_{1})$ (and thus contained in $B(x,R))$. There are at least $%
t_{1} $ disjoint balls centred in each $B(x_{j},r_{1}-r_{2})$ with radius $%
r_{2}$. These balls are each contained in the (disjoint) sets $%
B(x_{j},r_{1}),$ so that all $J\cdot t$ balls are disjoint. Hence if $k=2$ ($%
r_{2}=r$), then we have produced $J\cdot t_{1}$ disjoint balls of radius $r$
centred in $B(x,R)$ and that proves the result in this case. If $k>2$ we
repeat the construction.
\end{proof}
\end{lemma}

\begin{theorem}
Let $E$ be a doubling metric space. Then for any $\theta \in (0,1)$, 
\begin{equation*}
\dimLcantheta E=\inf_{0<\psi \leq \theta }\dimLscotpsi E.
\end{equation*}
Further, 
\begin{equation}  \label{eq:limit}
\dimqL E = \lim_{\theta\to 1} \dimLscotpsi E.
\end{equation}
\end{theorem}

\begin{remark}
The analogous result was proved for the Assouad dimension in \cite{FHHTY}
for subsets of $\mathbb{R}^{d},$ but the same proof applies in any doubling
metric space.
\end{remark}

\begin{proof}
Our argument is similar to the proof of the corresponding result for the
lower spectrum of measures. Let $s=\dimLcantheta E$. This dimension is
finite since the metric space $E$ is doubling and hence its upper Assouad
dimension is finite. Fix $\varepsilon >0$ and obtain $x_{i}\in E$, $%
R_{i}\rightarrow 0$ and $\theta _{i}\leq \theta $ such that 
\begin{equation}
M_{R^{1/\theta _{i}}}(B(x_{i},R_{i}))\leq R_{i}^{(1-1/\theta
_{i})(s+\varepsilon )}.  \label{LSpSet}
\end{equation}%
Without loss of generality, we can assume $\theta _{i}\rightarrow \psi $
where $\psi \in \lbrack 0,\theta ]$ and that the convergence is monotonic.

Case 1: We will first assume $\psi >0$. If $(\theta _{i})$ is an increasing
sequence, then $R^{1/\theta _{i}}\leq R^{1/\psi }$ and thus%
\begin{equation*}
N_{R^{1/\psi }}(B(x_{i},R_{i}))\leq N_{R^{1/\theta
_{i}}}(B(x_{i},R_{i}))\leq R_{i}^{(1-1/\theta _{i})(s+\varepsilon )}\leq
R_{i}^{(1-1/\psi )(s+\varepsilon /2)}\text{ for large }i.
\end{equation*}

Otherwise, $(\theta _{i})$ decreases to $\psi $. As each $\theta _{i}\leq
\theta <1,$ the same is true for $\psi $ and furthermore, $R_{1}^{1/\theta
_{i}}\geq R_{i}^{1/\psi }$. Let $D$ be the upper Assouad dimension of $E$.
For small enough $R_{i}$ we have%
\begin{equation*}
N_{R^{1/\psi }}(B(x_{i},R_{i}^{1/\theta _{i}}))\leq R_{i}^{(1/\theta
_{i}-1/\psi )(D+\varepsilon )},
\end{equation*}%
hence%
\begin{eqnarray*}
N_{R^{1/\psi }}(B(x_{i},R_{i})) &\leq &N_{R^{1/\theta
_{i}}}(B(x_{i},R_{i}))N_{R^{1/\psi }}(B(x_{i},R_{i}^{1/\theta _{i}})) \\
&\leq &R_{i}^{(1-1/\theta _{i})(s+\varepsilon )+(1/\theta _{i}-1/\psi
)(D+\varepsilon )}.
\end{eqnarray*}%
Since $\theta _{i}\rightarrow \psi $, and $\varepsilon >0$ was arbitrary, we
again deduce that $\dimAscotpsi E=s.$

Case 2: $\psi =0$. As in the proof of Theorem \ref{thm:spectrum}, choose $%
0<\beta <\theta <1$ with $\log \theta /\log \beta \notin \mathbb{Q}$ and
suppose for a contradiction that 
\begin{equation*}
\min \{\dimLscottheta(E),\dimLscotbeta(E)\}\geq s+3\varepsilon \text{.}
\end{equation*}%
This inequality implies that for all $x$, small enough $R$ and $\gamma
=\theta ,\beta ,$%
\begin{equation}
M_{R^{1/\gamma }}(B(x,R))\geq R^{(1-1/\gamma )(s+5\varepsilon /2)}
\label{**}
\end{equation}%
Let $\eta >0$ and choose $n,m$ as in Claim \ref{thm:numTheo}. Appealing to
Lemma \ref{packing} we see that 
\begin{equation}  \label{eq:*}
M_{R_{i}^{1/\theta ^{n}\beta ^{m}}}(B(x_{i},R_{i}))\geq M_{R_{i}^{1/\theta
}}(B(x_{i},R_{i}-R_{i}^{1/\theta })\cdot A_{i,n}\cdot B_{i,m}
\end{equation}%
where 
\begin{equation*}
A_{i,n}=\prod\limits_{k=1}^{n-1}\inf_{y_{k}}M_{R_{i}^{1/\theta
^{k+1}}}(B(y_{k},R_{i}^{1/\theta ^{k}}-R_{i}^{1/\theta ^{k+1}})
\end{equation*}%
and 
\begin{equation*}
B_{i,m}=\prod\limits_{k=1}^{m}\inf_{z_{k}}M_{R_{i}^{1/\theta ^{n}\beta
^{k}}}(B(z_{k},R_{i}^{1/\theta ^{n}\beta ^{k-1}}-R_{i}^{1/\theta ^{n}\beta
^{k}}).
\end{equation*}

The doubling condition combined with property \eqref{**} implies that there
is a constant $c>0$ such that whenever $1\leq \alpha \leq 2$, $\gamma
=\theta $ or $\beta $ and $R$ is small enough (depending only on $%
\varepsilon ,\gamma $ and $c$), then%
\begin{equation*}
M_{\alpha R^{1/\gamma }}(B(x,R))\geq cM_{R^{1/\gamma }}(B(x,R))\geq
cR^{(1-1/\gamma )(s+5\varepsilon /2)}\geq \left( \frac{R}{R^{1/\gamma }}%
\right) ^{s+2\varepsilon }.
\end{equation*}%
Since 
\begin{equation*}
R_{i}^{1/\theta ^{j-1}}-R_{i}^{1/\theta ^{j}}\leq R_{i}^{1/\theta ^{j}}\leq
2(R_{i}^{1/\theta ^{j-1}}-R_{i}^{1/\theta ^{j}})^{1/\theta }
\end{equation*}%
and 
\begin{equation*}
R_{i}^{1/\theta ^{n}\beta ^{j-1}}-R_{i}^{1/\theta ^{n}\beta ^{j}}\leq
R_{i}^{1/\theta ^{n}\beta ^{j}}\leq 2(R_{i}^{1/\theta ^{n}\beta
^{j-1}}-R_{i}^{1/\theta ^{n}\beta ^{j}})^{1/\beta },
\end{equation*}%
it follows that if we simplify the notation by putting 
\begin{equation*}
P_{i,j}=R_{i}^{1/\theta ^{j-1}}-R_{i}^{1/\theta ^{j}}\text{ and }%
Q_{i,j}=R_{i}^{1/\theta ^{n}\beta ^{j-1}}-R_{i}^{1/\theta ^{n}\beta ^{j}},
\end{equation*}%
then 
\begin{equation*}
M_{R_{i}^{1/\theta ^{n}\beta ^{m}}}(B(x_{i},R_{i}))\geq \left(
\prod\limits_{j=1}^{n}\frac{P_{i,j}}{P_{i,j}^{1/\theta }}\prod%
\limits_{j=1}^{m}\frac{Q_{i,j}}{Q_{i,j}^{1/\beta }}\right) ^{s+2\varepsilon
}.
\end{equation*}%
It is helpful to isolate the first term of the numerator together with the
last term of the denominator and then pair up the remaining terms giving the
expression%
\begin{eqnarray*}
&&M_{R_{i}^{1/\theta ^{n}\beta ^{m}}}(B(x_{i},R_{i})) \\
&\geq &\left( \frac{R_{i}-R_{i}^{1/\theta }}{(R_{i}^{1/\theta ^{n}\beta
^{m-1}}-R_{i}^{1/\theta ^{n}\beta ^{m}})^{1/\beta }}\prod\limits_{j=2}^{n}%
\frac{P_{i,j}}{P_{i,j-1}^{1/\theta }}\frac{R_{i}^{1/\theta
^{n}}-R_{i}^{1/\theta ^{n}\beta }}{(R_{i}^{1/\theta ^{n-1}}-R_{i}^{1/\theta
^{n}})^{1/\theta }}\prod\limits_{j=2}^{m}\frac{Q_{i,j}}{Q_{i,j-1}^{1/\beta }}%
\right) ^{s+2\varepsilon }.
\end{eqnarray*}%
Using a Taylor series expansion for $(1-x)^{1/\theta }$ for $x$ near $0,$
one can check that 
\begin{equation*}
\frac{P_{i,j}}{P_{i,j-1}^{1/\theta }},\frac{Q_{i,j}}{Q_{i,j-1}^{1/\beta }}%
\text{ and }\frac{R_{i}^{1/\theta ^{n}}-R_{i}^{1/\theta ^{n}\beta }}{%
(R_{i}^{1/\theta ^{n-1}}-R_{i}^{1/\theta ^{n}})^{1/\theta }}\geq 1.
\end{equation*}%
Hence we deduce that%
\begin{equation}  \label{eq:**}
M_{R_{i}^{1/\theta ^{n}\beta ^{m}}}(B(x_{i},R_{i}))\geq \left( \frac{%
R_{i}-R_{i}^{1/\theta }}{(R_{i}^{1/\theta ^{n}\beta ^{m-1}}-R_{i}^{1/\theta
^{n}\beta ^{m}})^{1/\beta }}\right) ^{s+2\varepsilon }\geq \left( \frac{R_{i}%
}{2R_{i}^{1/\theta ^{n}\beta ^{m}}}\right) ^{s+2\varepsilon }
\end{equation}
once $R_{i}$ is small enough. But since $R_{i}^{1/\theta _{i}}\leq
R_{i}^{1/\theta ^{n}\beta ^{m}}$ we have 
\begin{equation*}
M_{R_{i}^{1/\theta _{i}}}(B(x_{i},R_{i}))\geq M_{R_{i}^{1/\theta ^{n}\beta
^{m}}}(B(x_{i},R_{i})).
\end{equation*}%
Combining these observations with \eqref{LSpSet} gives the inequality%
\begin{equation*}
R_{i}^{(1-1/\theta _{i})(s+\varepsilon )}\geq cR_{i}^{(1-1/\theta ^{n}\beta
^{m})(s+2\varepsilon )}.
\end{equation*}%
But as we saw in the conclusion of the proof of Theorem \ref{thm:spectrum},
it is not possible for this inequality to be true for $R_{i}$ tending to
zero and that contradiction completes the first part of the proof.

To see that \eqref{eq:limit} holds, assume $\dimLscottheta = t$. Repeat the
argument starting at \eqref{eq:*} with $m=0$ to obtain 
\begin{equation*}
M_{R^{1/\theta^n}}(B(x,R)) \geq c \left( \frac{R}{R^{1/\theta^n}}
\right)^{t-\varepsilon}
\end{equation*}
as in \eqref{eq:**}, for all $x\in E$ and $R>0$ small enough. This shows $%
\underline{\dim}_{\,A}^{\,=\theta^n} E \geq t-\varepsilon$ for all $%
\varepsilon>0$ and so $\underline{\dim}_{\,A}^{={\theta}^n} E \geq %
\dimLscottheta E$. According to \cite[Theorem 3.10]{FY1} the function $%
\dimLscottheta E$ is continuous for $\theta\in(0,1)$. Following the argument
found in \cite[\S 3.2]{FHHTY}, $\lim_{\theta\to 1}\dimLscottheta E$ exists
and hence $\lim_{\theta\to 1}\dimLscottheta E= \liminf_{\theta\to 1}%
\dimLscottheta E =\inf_{\theta\in(0, 1)}\dimLscottheta E$.
\end{proof}

\end{document}